\documentclass[a4paper,reqno]{amsart}
\usepackage{latexsym,amsmath,amssymb,amsfonts}
\usepackage{color}

\usepackage{graphicx, caption}
\usepackage[usenames,dvipsnames]{pstricks}
\usepackage{float}
\usepackage{epsfig}
\usepackage{pst-grad} 
\usepackage{pst-plot} 
\usepackage{esint}
\usepackage{amssymb}

\newcommand{\R}{\mathbb R}
\newcommand{\N}{\mathbb N}
\newcommand{\np}{\mathcal{F}}
\newcommand{\id}{\mathrm{Id}}
\newcommand{\per}{\mathcal{P}}
\newcommand{\nl}{\mathcal{NL}}
\newcommand{\meno}{\!\setminus\!}
\renewcommand{\div}{{\rm div}}
\newcommand{\dx}{\;\mathrm{d}x}
\newcommand{\dy}{\;\mathrm{d}y}
\newcommand{\h}{\mathcal{H}^{N-1}}
\renewcommand{\dh}{\;\mathrm{d}\mathcal{H}^{N-1}}

\newcommand{\dist}{\mathrm{d}}
\renewcommand{\Pi}{\mathbb{T}}

\theoremstyle{plain}
\begingroup
\newtheorem{theorem}{Theorem}[section]
\newtheorem{lemma}[theorem]{Lemma}
\newtheorem{proposition}[theorem]{Proposition}
\newtheorem{corollary}[theorem]{Corollary}
\endgroup
\theoremstyle{definition}
\begingroup
\newtheorem{definition}[theorem]{Definition}
\newtheorem{remark}[theorem]{Remark}

\endgroup
\theoremstyle{remark}

\mathchardef\emptyset="001F

\numberwithin{equation}{section}

\newcommand{\measurerestr}{%
  \,\raisebox{-.127ex}{~\reflectbox{\rotatebox[origin=br]{-90}{$\lnot$}}}\,%
}

\title[Periodic critical points]{On periodic critical points  and local minimizers of the Ohta-Kawasaki functional}
\author{R.\ Cristoferi}

\begin{document}

\begin{abstract}
In this paper we collect some new observations about periodic critical points and local minimizers of a nonlocal isoperimetric problem arising in the modeling of diblock copolymers. In the main result, by means of a purely variational procedure, we show that it is possible to construct (locally minimizing) periodic critical points whose shape resemble that of any given strictly stable  constant mean curvature (periodic) hypersurface.
Along the way, we establish several auxiliary results of independent interest. 
\end{abstract}

\maketitle

\section{Introduction}

In this paper we study some properties of critical points of the functional
\begin{equation}\label{eq:npp}
\mathcal{F}^\gamma(E):=\per_{\Pi^N}(E)+\gamma\int_{\Pi^N}\int_{\Pi^N}G_{\Pi^N}(x,y)u^E(x)u^E(y)\dx\dy\,,
\end{equation}
where $\gamma\geq 0$, $E$ is a subset of the $N$-dimensional flat torus $\Pi^N$, $N\geq 2$, $\per_{\Pi^N}(E)$ denotes the perimeter of $E$ in $\Pi^N$,
$u^E(x):=\chi_E(x)-\chi_{\Pi^N\setminus E}(x)$, and, for every $x\in\Pi^N$, $G_{\Pi^N}(x,\cdot)$ is the unique solution of
$$
-\triangle_y G_{\Pi^N}(x,\cdot)=\delta_x(\cdot)-1\quad\text{ in } \Pi^N,\quad\quad\quad\ \int_{\Pi^N}G_{\Pi^N}(x,y)\dy=0\,.
$$
We will refer to the first term of \eqref{eq:npp} as the \emph{local term}, while to the second one as the \emph{nonlocal term}. The latter will be denoted with $\gamma\nl(E)$.
We notice that the local term favors the formation of large regions of pure phase, while the nonlocal one prefers the function $u^E$ to oscillate (see Remark \ref{rem:nonloc}).

The functional \eqref{eq:npp} arises as the variational limit (in the sense of $\Gamma$-convergence) of the $\varepsilon$-diffuse Ohta-Kawasaki energy
\begin{align}\label{eq:ok}
OK_\varepsilon(u):=\varepsilon \int_\Omega &|\nabla u|^2\mathrm{d}x + \frac{1}{\varepsilon}\int_\Omega (u^2-1)^2\mathrm{d}x \nonumber\\
&+ \gamma \int_\Omega\int_\Omega G(x,y) \bigl( u(x)-m \bigr)\bigl( u(y)-m \bigr)\mathrm{d}x\mathrm{d}y\,,
\end{align}
where $\Omega\subset\R^N$ is an open set, $G$ is the Green's function for $-\triangle$, $u\in H^1(\Omega)$, and $m:=\fint_\Omega u$.
The functional $OK_\varepsilon$ has been introduced by Ohta and Kawasaki in \cite{OhtaKaw} to model microphase separation of a class of two-phase materials called diblock copolymers (see \cite{ChoRen} for a rigorous derivation of the Ohta-Kawasaki energy from first principles, and \cite{Mur} for a physical background on long-range interaction energies). These materials are linear-chain macromolecules, each consisting of two thermodynamicalLy incompatible subchains joined covalently, that correspond to the regions where $u\approx-1$ and $u\approx+1$ respectively.
Due to this incompatibility, the two phases try to separate as much as possible; on the other hand, because of the chemical bonds, only partial separation can occur at a suitable mesoscale.
Such a partial segregation of these chains produces very complex patterns, that are experimentally observed to be (quasi) periodic at an intrinsic scale.
The structure of these patterns depends strongly on the volume fraction of a phase with respect to the other, but they are seen to be very close to \emph{periodic surfaces with constant mean curvature} (see Figure ~\ref{fig:pattern}).

\begin{figure}[H]
\includegraphics[scale=0.25]{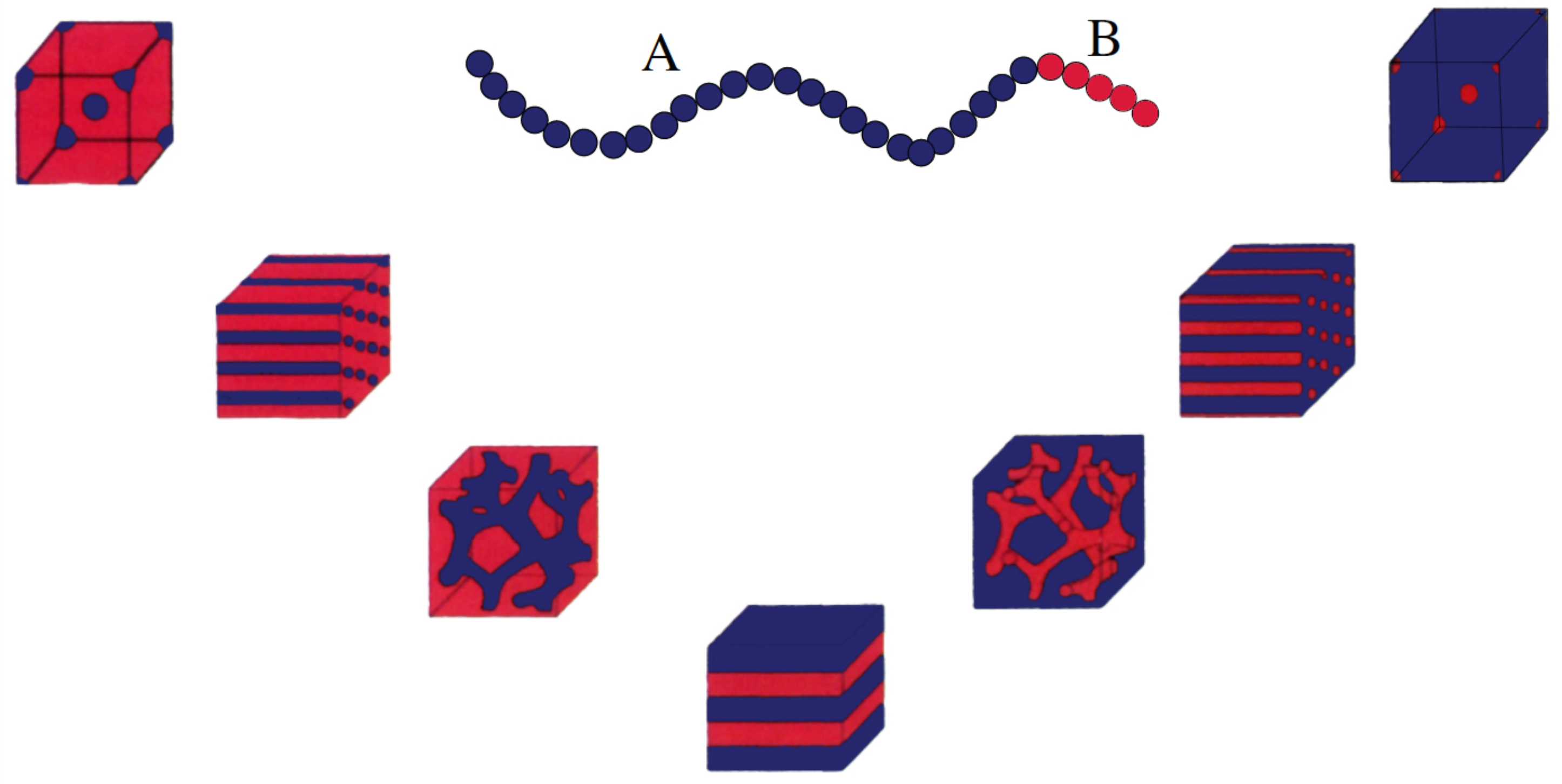}
\caption{The typical patterns that are observed according to an increasing value of the volume fraction.}\label{fig:pattern}
\end{figure}

According to the theory proposed by Ohta and Kawasaki in \cite{OhtaKaw}, we expect observable configurations to be global (or local) minimizers of the energy \eqref{eq:ok}. Since the parameter $\varepsilon$ is usually small, from the mathematical point of view it is more convenient to consider the variational limit of the energy $OK_\varepsilon$ that, in the periodic setting, turns out to be the sharp interface energy \eqref{eq:npp}.

Proving analitically that global minimizers of \eqref{eq:npp} or \eqref{eq:ok} are (quasi) periodic is a formidable task. In the case $N=1$, the periodicity of minimizers has been established by Ren and Wei in \cite{RenWei5} (see also the work by M\"{u}ller, \cite{M}), but in the case of higher dimensions it is still an open problem.
Indeed, so far, the best result in this direction is the work \cite{ACO} by Alberti, Choksi and Otto, where it is proved that global minimizers of \eqref{eq:npp} in big cubes under a volume constraint present a quasi uniform energy distribution of each component of the energy.
This result has been extended to the case of the functional \eqref{eq:ok} by Spadaro in \cite{Spad}.
Albeit the above conjecture is still open, the structure of global minimizers has been investigated by many authors (see, for example, \cite{ChoPel1, ChoPel2, CicSpa, GolMurSer1, GolMurSer2, MorSte, Mur2, Osh, SteTop, Top}), but only in some asymptotic regimes, \emph{i.e.}, when the parameter $\gamma$ is small or $m\approx\pm1$.

A more reasonable, but still highly nontrivial, purpose is to exhibit a class of local minimizers of the energies \eqref{eq:npp} and \eqref{eq:ok} that look like the observed configurations. Among the results in this direction we would like to recall the works by Ren and Wei (\cite{RenWei, RenWei1, RenWei2, RenWei3, RenWei4}), where they construct explicit critical configurations of the sharp interface energy, with lamellar, cylindrical and spherical patterns. They also provide a regime of the parameters that ensures the (linear) stability of such configurations. The natural notion of stability for \eqref{eq:npp} was introduced by Choksi and Strernberg in \cite{ChoSte} (see also \cite{Mur} for a formal derivation of the first and second variation of the functional), and it has been subsequently proved by Acerbi, Fusco and Morini in \cite{AFM}, that critical and strictly stable (namely with strictly positive second variation) configurations are local minimizers in the $L^1$ topology.\\

The aim of our work is to collect some new observations on critical points of the sharp interface energy \eqref{eq:npp}.  

We start by showing, in Proposition~\ref{prop:small},  that critical points are always local minimizers with respect to perturbations with sufficiently small support. This minimality-in-small-domains property of critical points is shared by many functionals of the Calculus of Variations, 
but to the best of our knowledge it has been never been observed before for the Ohta-Kawasaki energy.

The second result (see Proposition~\ref{prop:unif}) shows  that the property of being critical and stable is preserved under small perturbations of the parameter $\gamma$. More precisely, we show that, given $\bar\gamma\geq 0$ and a strictly stable critical point $E$ of the functional
$\mathcal{F}^{\bar\gamma}$, we can find a (unique) family $(E_\gamma)_\gamma$ of smoothly varying uniform local minimizers of $\mathcal{F}^{\gamma}$ for $\gamma$ ranging in a small neighborhood of $\bar\gamma$. 
The procedure to construct such a family is purely variational and based on showing that the local minimality criterion provided in \cite{AFM} can be made uniform with respect to the parameter $\gamma$ and with respect to critical sets ranging in a sufficiently small  $C^1$-neighborhood of a given strictly stable set $E$. Such an observation, which has an independent interest, is proven in Proposition~\ref{prop:unif}. We would like to remark that, in proving the above result, we in fact drastically simplify the general argument, by replacing \cite[Lemma 3.8]{AFM} with a penalization argument, that was inspired to us by \cite{DamLam}.

The above stability property is used to establish the main result of this paper (see Theorem~\ref{thm:main}): given $\bar\gamma>0$  and $\varepsilon>0$
and a subset $E$ of the torus $\Pi^N$ such that $\partial E$ is a strictly stable constant mean curvature hypersurface, we show that it is possible to  find an integer $k = k(\bar\gamma,\varepsilon)$ and a $1/k$-periodic critical point of $\mathcal{F}^{\bar\gamma}_{\Pi^N}$, whose shape is 
$\varepsilon$-close (in a $C^1$-sense) to the $1/k$-rescaled version of $E$ and whose mean curvature is almost constant.
Moreover, such a critical point is an isolated local minimizer with respect to $(1/k)$-periodic perturbations. 
In words, the above result says that it is possible to construct local minimizing periodic critical points of the energy \eqref{eq:ok}, whit a shape closely resembling that of any given strictly stable periodic constant mean curvature surface.


An important consequence of our variational procedure is that it allows to show (see Proposition~\ref{prop:approx}) that all the constructed critical points can be approximated by critical points of the $\varepsilon$-diffuse energy \eqref{eq:ok}. This is done by usign a $\Gamma$-convergence argument in the spirit of the Kohn and Sternberg theory, see \cite{KohSter}.

We conclude by remarking that numerical and experimental evidences suggest the following general structure for global minimizers: the nonlocal term determines an intrinsic scale of periodicity (the larger is  $\gamma$  the smaller is the periodicity scale), while the shape of the global minimizer inside the periodicity cell is dictated by the perimeter term. Although we are  very far from an analytical validation of such a picture, our result allows to construct a class of (locally minimizing) critical point that display the above structure. \\


\section{Preliminaries}

Given $k\in\N\setminus\{0\}$, we will denote by $\Pi^N_k$ the $N$-dimensional flat torus rescaled by a factor $1/k$, \emph{i.e.}, the quotient of
$\R^N$ under the equivalence relation
$$
\widehat{x} \sim_k \widehat{y} \quad\Leftrightarrow\quad k(\widehat{x}-\widehat{y}) \in \mathbb{Z}^N\,.
$$
Hereafter we will denote $\Pi^N_1$ by $\Pi^N$.
Points in $\Pi^N_k$ will be denoted by $x$, $y$.
A set $F\subset\Pi_k^N$ can be naturally identified with the $1/k$-periodic set of $\R^N$ (or of $\Pi^N$) that equals (a translate of) $F$ in each $1/k$-periodicity cell (see Figure \ref{fig:per} on the right).
When we speak about the regularity of a set $F\subset\Pi^N_k$, we will always refer to the regularity of the $1/k$-periodic set $F\subset\R^N$.
Finally, for $\beta\in(0,1)$ and $r\in\N$, we define the functional space $C^{r,\beta}(\Pi^N_k)$ as the space of $1/k$-periodic functions in $C^{r,\beta}(\R^N)$.

We now recall some geometric definitions: given a set $E\subset\Pi^N$ of class $C^2$, we will denote by $D_\tau$ the tangential gradient operator, by $\div_{\tau}$ the tangential divergence, by $\nu_E$ the normal vector field on $\partial E$, by $B_{\partial E}$ its second fundamental form, and by $|B_{\partial E}|^2$ its Euclidean norm, that coincides with the sum of the squares of the principal curvatures of $\partial E$. Finally, $H_{\partial E}$ denotes the sum of the principal curvatures of $\partial E$.

We are now in position to introduce the main object we will need in the paper.

\begin{definition}\label{def:per}
Given a set $E\subset\Pi^N$ and $k\in\N\setminus\{0\}$, we define the set $E^k\subset\Pi^N_k$ as follows:
$$
E^k:=\{ x\in\Pi^N_k \;:\; kx\in E \}\,.
$$
\end{definition}

\begin{figure}[H]
\begin{center}
\includegraphics[scale=2]{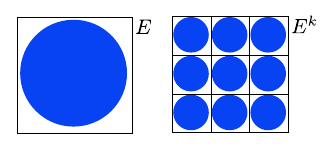}
\caption{A set $E\subset\Pi^N$ on the left, and the set $E^k$, with $k=3$, seen as a subset of $\Pi^N$, on the right.}
\label{fig:per}
\end{center}
\end{figure}

\begin{remark}
Notice that $\int_{\Pi^N}u^E\dx=\fint_{\Pi^N_k}u^{E_k}_k\dx$, where $u^F_k:=\chi_F-\chi_{\Pi^N_k\meno F}$.
\end{remark}

We now introduce the notion of perimeter in $\Pi^N_k$.

\begin{definition}
Let $E\subset\Pi^N_k$. We say that $E$ is a set of \emph{finite perimeter} in $\Pi^N_k$ if
$$
\sup\Bigl\{ \int_{E} \div\,\xi\dx \;:\; \xi\in C^1(\Pi^N_k;\R^N)\,,|\xi|\leq1 \Bigr\}<\infty\,.
$$
In this case we denote by $\per_k(E)$ the above quantity.
\end{definition}

We now introduce two ways for measuring how close two sets in $\Pi^N$ are.
The first one is an $L^1$ distance that takes into account the fact that our functional is translation invariant.

\begin{definition}\label{def:alpha}
Given two sets $E, F\subset\Pi^N_k$, we set:
$$
\alpha(E, F):=\min_{x\in\Pi^N_k}|E\triangle(x+F)|\,.
$$
\end{definition}

In the following we will also consider sets whose boundary is a normal graph over the boundary of another set.
Thus, we also need to measure how close such two sets are.

\begin{definition}
Let $E\subset\Pi^N_k$. Then, for sets $F\subset\Pi^N_k$ such that
$$
\partial F=\{ x+\psi(x)\nu_E(x)\;:\; x\in \partial E \}\,,
$$
for some function $\psi\in C^{r,\beta}(\partial E)$, we define
$$
\dist_{C^{r,\beta}}(E, F):=\|\psi\|_{C^{r,\beta}(\partial E)}\,.\\
$$
\end{definition}


\subsection{The area functional}

We recall some results about the area functional.

\begin{definition}
We say that a set $E\subset\Pi^N_k$ is a \emph{local minimizer of the area functional} if there exists $\delta>0$ such that
$$
\per_k(E)\leq \per_k(F)\,,
$$
for all $F\subset\Pi^N_k$ with $|E|=|F|$, such that $\alpha(E,F)\leq\delta$.
\end{definition}

\begin{definition} \label{def:quamin1}
A set $E\subset\Pi^N_k$ is said to be an \emph{$(\omega,r_0)$-minimizer} for the area functional, with $\omega>0$ and $r_0>0$,
if for every ball $B_r(x)$ with $r\leq r_0$ we have
$$
\per_k(E) \leq \per_k(F) + \omega |E\triangle F|,
$$
whenever $F\subset\Pi^N_k$ is a set of finite perimeter such that $E\triangle F\subset\subset B_r(x)$.
\end{definition}

We recall an improved convergence theorem for $(\omega,r_0)$-minimizers of the area functional. This result is well-known to the experts (see, for istance, \cite{Whi}). For a complete find see, for instance, \cite{CicLeo}.

\begin{theorem} \label{thm:W}
Let $(E_n)_n$ be a sequence of $(\omega,r_0)$-minimizers of the area functional such that
$$
\sup_n\per_k(E_n)<+\infty
\quad\text{and}\quad
\alpha(E_n, E)\rightarrow0\text{ as }n\rightarrow\infty\,,
$$
for some bounded set $E$ of class $C^2$. Then, for $n$ large enough, $E_n$ is of class $C^{1,\beta}$ for all $\beta\in(0,1)$, and
$$
\partial E_n = \{x+\psi_n(x)\nu_E(x) \,:\, x\in\partial E \},
$$
with $\psi_n\to0$ in $C^{1,\beta}(\partial E)$ for all $\beta\in(0,1)$.\\
\end{theorem}


\subsection{The functional $\np_k^\gamma$}\label{sec:funct}

We first define the functionals we are interested in.

\begin{definition}
Given $\gamma\geq 0$ and $k\in\N$, we define, for sets $E\subset\Pi^N_k$, the functional
\begin{align}\label{eq:effe}
\np_k^\gamma(E)&:=\per_k(E)+\gamma\nl_k(E) \nonumber \\
&:=\per_k(E)+\gamma\int_{\Pi^N_k}\int_{\Pi^N_k}G_k(x,y)u_k^E(x)u_k^E(y)\dx\dy\,,
\end{align}
where $u_k^E(x):=\chi_E(x)-\chi_{\Pi^N_k\setminus E}(x)$ and $G_k$ is the unique solution of
$$
-\triangle_y G_k(x,\cdot)=\delta_x(\cdot)-\frac{1}{|\Pi^N_k|}\quad\text{ in } \Pi^N_k,\quad\quad\quad \int_{\Pi^N_k}G_k(x,y)\dy=0\,.
$$
For simplicity, we will denote by $\np^\gamma$ and $u^E$ the functional $\np_1^\gamma$ and the function $u_1^E$ respectively.
\end{definition}

\begin{remark}
Notice that the area functional corresponds to the choice of $\gamma=0$.
\end{remark}

We now introduce the main objects under investigation in this paper: critical points and local minimizers.

\begin{definition}
A set $E\subset\Pi^N$ of class $C^2$ will be called \emph{critical} for the functional $\np^\gamma$ if there exists a constant $\lambda\in\R$ such that
$$
H_{\partial E} + 4\gamma v^E=\lambda\quad\quad\text{ on } \partial E\,,
$$
where
$$
v^E(x):=\int_{\Pi^N}G(x,y)u^E(y)\dy\,.
$$
\end{definition}

\begin{remark}
The above definition is motivated by the fact that (as one expects) the first variation of the functional $\np^\gamma$ vanishes on critical sets (see Theorem~\ref{thm:var}).
\end{remark}

\begin{definition}
We say that a set $E\subset\Pi^N_k$ is a \emph{local minimizer} of the the functional $\np_k^\gamma$, if there exists $\delta>0$ such that
$$
\np_k^\gamma(E)\leq \np_k^\gamma(F)\,,
$$
for all $F\subset\Pi^N_k$ with $|E|=|F|$, such that $\alpha(E,F)\leq\delta$. Moreover, we say that $E$ is an \emph{isolated local minimizer} if the above inequality is strict whenever $\alpha(E,F)>0$.
\end{definition}

We now want to derive some regularity properties of local minimizers of $\np_k^\gamma$.
In order to do this, we observe that local minimizers of $\np_k^\gamma$ are in fact $(\omega,r)$-minimizer, and then we will rely on the well-known regularity theory for $(\omega,r)$-minimizers.

First of all one can see that the nonlocal term turns out to be Lipschitz (see \cite[Lemma 2.6]{AFM} for a proof).

\begin{proposition}[Lipschitzianity of the nonlocal term] \label{prop:lip}
There exists a constant $c_0$, depending only on $N$, such that
if $E,F\subset\Pi^N_k$ are measurable sets, then
\begin{equation*}
|\nl_k(E)-\nl_k(F)| \leq c_0\alpha(E,F)\,.
\end{equation*}
\end{proposition}

The following lemma is a refinement of a result already present in \cite{AFM} and \cite{EspFus}. The proof we present here follows the lines of those presented in \cite{Mor2}.

\begin{lemma}\label{lem:qm}
Fix constants $\bar{\gamma}>0$, $\delta_0>0$, $m_0\in(0,|\Pi^N_k|)$ and $M>0$. Let $E\subset\Pi^N_k$, with $\per_k(E)\leq M$, be a solution of 
\begin{equation}\label{eq:min}
\min\Bigl\{ \per_k(F)+\gamma\nl_k(F) \;:\; \fint_{\Pi^k} u_k^F=m\,,\,\alpha(E,F)\leq\delta \Bigr\}\,,
\end{equation}
where $\gamma\leq\bar{\gamma}$, $\delta\in[\delta_0,+\infty]$ and $m\in[-m_0, |\Pi^N_k|-m_0]$.
Then, we can find a constant
$\Lambda_0=\Lambda_0(c_0,m_0,\bar{\gamma},\delta_0,M)>0$ (where $c_0$ is the constant given by Proposition ~\ref{prop:lip}) such that $E$ is a solution of the \emph{unconstrained minimum problem}
$$
\min\Bigl\{ \per_k(F)+\gamma\nl_k(F)+\Lambda\Bigl| \fint_{\Pi^k} u_k^F-m \Bigr| \;:\; \alpha(E,F)\leq\delta/2 \Bigr\}\,,
$$
for all $\Lambda\geq\Lambda_0$.
\end{lemma}

\begin{proof}
The idea is to prove that we can find a constant $\Lambda_0$ as in the statement of the lemma such that if $\widetilde{F}\in\Pi^k$ solves
$$
\min\Bigl\{ \per_k(F)+\gamma\nl_k(F)+\Lambda\Bigl| \fint_{\Pi^k} u_k^F-m \Bigr| \;:\; \alpha(E,F)\leq\delta/2 \Bigr\}\,,
$$
where $\gamma\leq\bar{\gamma}$ and $\Lambda\geq\Lambda_0$, then $\alpha(\widetilde{F},E)=0$, where $E$ is a solution of \eqref{eq:min}. To prove it, suppose for the sake of contradiction that there exist sequences $\gamma_n\leq\gamma$, $\Lambda_n\rightarrow\infty$, sets $E_n$ solutions of
$$
\min\Bigl\{ \per_k(F)+\gamma_n\nl_k(F) \;:\; \fint_{\Pi^k} u_k^F=m_n\,,\,\alpha(E,F)\leq\delta_n \Bigr\}\,,
$$
where $\delta_n\geq\delta_0$, $m_n:=\fint_{\Pi^k} u_k^{E_n}\in[-m_0,|\Pi^N_k|-m_0]$, $\per_k(E_n)\leq M$, and sets $F_n$ solutions of
$$
\min\Bigl\{ \per_k(F)+\gamma_n\nl_k(F)+\Lambda_n \Bigl| \fint_{\Pi^k} u_k^F-m_n \Bigr| \;:\; \alpha(E_n,F)\leq\delta_n/2 \Bigr\}\,,
$$
with $m_n\neq\int_{\Pi^N_k}u_k^{F_n}$ (suppose $\int_{\Pi^N_k}u_k^{F_n}<m_n$). From now on we will suppose $|F_n\triangle E_n|=\alpha(E_n,F_n)$. The idea is to modify the sets $F_n$'s in such a way that $\int_{\Pi^N_k}u_k^{F_n}=m_n$ (notice that, since we are not working in the whole $\R^N$, we cannot just rescale them to fit the desired volume). This idea has been developed in \cite{EspFus}. Set
$$
\widetilde{\np}_n(F):=\np_k^{\gamma_n}(F)+\Lambda_n \Bigl| \fint_{\Pi^k} u_k^F-m \Bigr|\,.
$$
First of all we notice that $\sup_n\per_k(F_n)<\infty$. Indeed
\begin{align*}
\per_k(F_n)+\Lambda_n \Bigl| \fint_{\Pi^k} u_k^{F_n}-m_n \Bigr|&\leq \widetilde{\np}_n(E_n)-\gamma_n\nl_k(F_n)\\
&=\per_k(E_n)+\gamma_n\bigl( \nl_k(E_n)-\nl_k(F_n) \bigr)\leq M+\bar{\gamma}\delta c_0\,.
\end{align*}
Thus, up to a not relabelled subsequence, it is possible to find a set $F_0\subset\Pi^N_k$ with $\fint_kv_k^{F_0}\in[-m_0,|\Pi^N_k|-m_0]$, such that
$F_n\rightarrow F_0$ in $L^1$. Moreover $\alpha(E_n,F_n)\rightarrow0$. We now sketch the argument presented in \cite{EspFus}. Given $\varepsilon>0$, it is possible to find a radius $r>0$ such that (up to translations)
$$
|F_n\cap B_{r/2}|\leq\varepsilon r^N\,,\quad\quad |F_n\cap B_r|\geq \frac{\omega_Nr^N}{2^{N+2}}\,,
$$
for $n$ sufficiently large. Let $\sigma_n\in(0,1/2^N)$, that will be chosen later, and define
$$
\Phi_n(x):=
\left\{
\begin{array}{lll}
(1-\sigma_n(2^N-1))x & & \text{ if } |x|\leq\frac{r}{2}\,,\\
x+\sigma_n\bigl( 1-\frac{r^N}{|x|^N} \bigr)x & & \text{ if } \frac{r}{2}\leq|x|<r\,,\\
x & & \text{ if } |x|\geq r\,.
\end{array}
\right.
$$
Let $\widetilde{F}_n:=\Phi_n(F_n)$. It is possible to prove that
$$
\per_k(F_n\cap B_r)-\per_k(\widetilde{F}_n\cap B_r)\geq-2^N N\sigma_n\per_k(F_n\cap B_r)\,,
$$
and that, for $\varepsilon>0$ sufficiently small,
$$
\fint_{\Pi^N_k} u_k^{\widetilde{F}_n}-\fint_{\Pi^N_k} u_k^{F_n}\geq
       \sigma_n r^N\Bigl[ c\frac{\omega_N}{2^{N+2}}-\varepsilon\bigl(c+(2^N-1)N \bigr) \Bigr]\geq c\sigma_n r^N\frac{\omega_N}{2^{N+3}}
       =:C_1\sigma_nr^N\,,
$$
where $c$ and $C_1$ are constants depending only on the dimension $N$. Then it is possible to choose the $\sigma_n$'s in such a way that $|F_n|=|E_n|$ for all $n$. In particular we obtain, from the above inequality, that $\sigma_n\rightarrow0$. Finally, it is also possible to prove that
$$
\alpha(\widetilde{F}_n,F_n)\leq C_2\sigma_n\per_k(F_n\cap B_r)\,.
$$
Combining all these estimates we have that
$$
\widetilde{\np}_n(\widetilde{F}_n)\leq \widetilde{\np}_n(F_n)+\sigma_n\bigl[ (2^N N+C_2c_0\bar{\gamma})\per_k(F_n\cap B_r)-\Lambda_nC_1r^N \bigr]< 
             \widetilde{\np}_n(F_n)\leq \widetilde{\np}_n(E_n)\,.
$$
Since $\sigma_n\rightarrow0$, we have that, for $n$ large enough, $\alpha(\widetilde{F}_n,E_n)\leq\delta_n$. Thus the above inequality is in contradiction with the local minimality property of $E_n$.
\end{proof}

\begin{corollary}\label{cor:qm}
Let $E\subset\Pi^N_k$ be a local minimizer of $\np_k^\gamma$. Then it holds that $E$ is an $(\omega,r)$-minimizer of the area functional.
Moreover the parameter $\omega$ depends on the constants $c_0, m_0, \bar{\gamma}, \delta_0$ and $M$ of the previous lemma.
\end{corollary}

\begin{proof}
From the above result, it follows that local minimizers of $\np_k^\gamma$ are in fact $(\omega,r)$-minimizer, providing we take $\omega:=c_0+\Lambda$ and we choose $r>0$ such that $\omega_N r^N\leq\delta/2$.
\end{proof}

The regularity theory for $(\omega,r)$-minimizers allows us to say something about the regularity of local minimizers of $\np_k^\gamma$ (see \cite[Theorem 1]{Tam2}).

\begin{proposition}
Let $E\subset\Pi^N_k$ be a local minimizer of $\np_k^\gamma$. Then we can write $\partial E=\partial^*E\cup\Sigma$, where the reduced boundary $\partial^* E$ is of class $C^{3,\alpha}$ for all $\alpha\in(0,1)$, and the Hausdorff dimension of $\Sigma$ is less than or equal to $N-8$. 
\end{proposition}

\begin{remark}
Using the equation satisfied by a critical set $E$, it is also possible to prove (see \cite{JulPis}) the $C^{\infty}$ regularity of $\partial^* E$, in every dimension $N$. In particular, in dimension $N\leq 7$, we obtain the $C^{\infty}$-regularity for the entire boundary $\partial E$.
\end{remark}

In the remaining part of this section we would like to investigate some properties of the nonlocal term, as well as the relation between the functionals $\np$ and $\np_k$. In particular, we would like to establish the relation between the followings
\[
\np^\gamma(E)\,,\quad\quad\np^\gamma_k(E^k)\,,\quad\quad\np^\gamma(E^k)\,,
\]
where in the last, $E^k$ is seen as a subset of $\Pi^N$, \emph{i.e.}, as $k^N$ copies of $E/k$.
The relation between the perimeter terms is clear, since we have a scaling for it.
The behavior of the three nonlocal terms above can be well understood once we introduce the following function

\begin{definition}
For a set $E\subset\Pi^N_k$, we define, for $x\in\Pi^N_k$, the function
$$
v_k^E(x):=\int_{\Pi^N_k}G_k(x,y)u_k^E(y)\dy\,.
$$
For simplicity, we wil denote the function $v^E_1$ by $v^E$.
\end{definition}

\begin{remark}\label{rem:nonloc}
Notice that, for $E\in\Pi^N_k$, $v_k^E$ is the unique solution to
\begin{equation}\label{eq:vE}
-\triangle v_k^E=u_k^E-m^E\quad\text{ in } \Pi^N_k\,,\quad\quad\quad\int_{\Pi^N_k}v_k^E\dx=0\,,
\end{equation}
where we recall that $m^E:=\int_{\Pi^N}u^E\dx=\fint_{\Pi^N_k} u_k^{E^k}\dx$. Moreover, one can see that $v_k^E$ is $1/k$-periodic.
Thus, it is possible to rewrite the nonlocal term in the following way:
$$
\nl_k(E)=\int_{\Pi^N_k}u_k^E v_k^E\dx=-\int_{\Pi^N_k}v_k^E\triangle v_k^E\dx=\int_{\Pi^N_k}|\nabla v_k^E|^2\dx\,.
$$
In particular, from the above writing, we see that the nonlocal term prefers highly oscillating functions $u^E_k$, as has been pointed out in the introduction.

By standard elliptic regularity we know that $v_k^E\in W^{2,p}(\Pi^N_k)$ for all $p\in[1,+\infty)$.
In particular, it holds that
$$
\|v_k^E\|_{W^{2,p}(\Pi^N_k)}\leq C\,,
$$
where $p>1$ and $C>0$ is a constant depending only on $\Pi^N_k$.
\end{remark}

We can now prove the relation between the three objects above.

\begin{lemma}\label{lem:Ek}
Let $E\subset\Pi^N$. Then it holds
\begin{equation}\label{eq:per}
\np^\gamma(E^k) = k^N \np_k^\gamma(E^k)\,,
\end{equation}
and
\begin{equation}\label{eq:pfk}
\np_k^\gamma(E^k)=k^{1-N}\Bigl[\per_{\Pi^N}(E)+\gamma k^{-3}\nl_{\Pi^N}(E)\Bigr]\,.
\end{equation}
\end{lemma}

\begin{proof}
We first prove \eqref{eq:pfk}. We claim that, for a set $E\subset\Pi^N$, we have
$$
v_k^{E^k}(x)=k^{-2}v^E(kx)\,.
$$
Indeed, noticing that $\fint_{\Pi^N_k}u_k^{E_k}=\fint_{\Pi^N}u^{E}$, it holds
$$
-\triangle\bigl(k^{-2}v^E(kx)\bigr) = -\triangle v^E(kx) = u^E(kx)-m = u_k^{E^k}(x)-m\,,
$$
and
$$
\int_{\Pi^N_k}k^{-2}v^E(kx)\dx = k^{-N-2}\int_{\Pi^N}v^E(y)dy=0\,.
$$
By uniqueness of the solution of problem \eqref{eq:vE}, we obtain our claim.
Finally, we can conclude by noticing that
$$
\int_{\Pi^N_k}|\nabla v_k^{E^k}(x)|^2\dx=k^{-2-N}\int_{\Pi^N}|\nabla v^E(x)|^2\dx\,.
$$

To prove \eqref{eq:per}, we just notice that $v^{E^k}$ is $1/k$-periodic.
\end{proof}

\begin{remark}
We would like to stress the meaning of the above result: given a set $E\subset\Pi^N$, equation \eqref{eq:per} tells us that
the energy of $E^k$ in $\Pi^N$ is just the sum of the energies of each of its \emph{pieces} in each $\Pi^N_k$,
each of which is given by formula \eqref{eq:pfk}. In particular, we have that $v^{E^k}$ is a $1/k$-periodic function that, in each periodicity cell, is equal to $v^E$ opportunely rescaled.
\end{remark}


\subsection{Results about $\Gamma$-convergence}

In this section we would like to recall an approximation theorem for isolated local minimizers of the area functional.
For, we need to write the functional $\np_{\Pi^N}^\gamma$ in the language of $\Gamma$-convergence.

\begin{definition}\label{def:gamma}
Let $(X,\mathrm{d})$ be a metric space, and let $F,\, F_n: X\rightarrow\R\cup\{+\infty\}$. We say that the sequence $F_n$ \emph{$\Gamma(\mathrm{d})$-converges} to the functional $F$ if the following two conditions are satisfied
\begin{itemize}
\item for every $x_n\stackrel{\mathrm{d}}{\rightarrow}x$, $F(x)\leq\liminf_{n}F_n(x_n)$,
\item for every $\bar{x}\in X$ there exists $x_n\stackrel{\mathrm{d}}{\rightarrow}\bar{x}$ such that $F(x)\geq\limsup_n F_n(x_n)$.
\end{itemize}
In this case we will write $F_n\stackrel{\Gamma(d)}{\rightarrow} F$.
\end{definition}

\begin{definition}
Consider the quotient space space $X:=L^1(\Pi^N)/\sim$, where the equivalence relation $\sim$ is defined as follows:
$f_1\sim f_2$ if and only if there exists $v\in\Pi^N$ such that $f_1(x+v)=f_2(x)$, for each $x\in\Pi^N$.
Endow this space with the distance
$$
\alpha(u,v):=\min_{x\in\Pi^N}\|u-v(\cdot-x)\|_{L^1(\Pi^N)}\,.
$$
Fix $\gamma\in[0,+\infty)$ and $m\in(-1,1)$ and define the functional $\widetilde{\np}_\gamma:X\rightarrow\R\cup\{+\infty\}$ as
$$
\widetilde{\np}^\gamma(u):=
\left\{
\begin{array}{ll}
\np^\gamma(E) & \text{ if } u=u^E, \text{ for some set } E \text{ with } \fint_{\Pi^N}u^E\dx=m\,,\\
+\infty & \text{ otherwise}\,.
\end{array}
\right.
$$
\end{definition}

\begin{remark}\label{rem:funct}
Notice that the functionals $\widetilde{\np}^\gamma$ turn out to be equi-coercive and lower semicontinuous. Moreover $\widetilde{\np}^\gamma\stackrel{\Gamma(\alpha)}{\longrightarrow}\widetilde{\np}^0$ as $\gamma\rightarrow0^+$.
\end{remark}

Although the $\Gamma$-convergence has been designed for the convergence of \emph{global} mininimizers, one can say also something about convergence of \emph{local} minimizers. The following result is a particular application of \cite{KohSter}.

\begin{theorem}\label{thm:KS}
Let $E\subset\Pi^N$ be a smooth isolated local minimizer of $\np^{\bar{\gamma}}$, for some $\bar{\gamma}\geq0$.
Then there exists a sequence $(E_\gamma)_{\gamma>0}$, with $|E_\gamma|=|E|$, such that $E_\gamma$ is a local minimizer of $\np^\gamma$ in $\Pi^N$ and
$\alpha(E_\gamma, E)\rightarrow0$ as $\gamma\rightarrow\bar{\gamma}$.
\end{theorem}


\section{Variations and local minimality}

In the following we will use a local minimality criterion provided in \cite{AFM}, that we recall here for reader's convenience. This criterion is based on the positivity of the second variation.
Thus, we need to introduce what we mean by variation.

\begin{definition}\label{def:adm}
Let $E\subset\Pi^N$ be a set of class $C^2$. Take a smooth vector field $X\in C^\infty(\Pi^N;\R^N)$ and consider the associated flow
$\Phi:\Pi^N\times(-1,1)\rightarrow\Pi^N$ given by
$$
\frac{\partial\Phi}{\partial t}=X(\Phi)\,,
$$
such that $\Phi(x,0)=x$ for all $x\in\Pi^N$. Let $E_t:=\Phi(E,t)$ and suppose $|E_t|=|E|$ for each time $t$.
We say that $(\Phi(\cdot,t))_t$ is an \emph{admissible family of diffeomorphisms} for $E$.
We define the \emph{first} and the \emph{second variation} of $\np^\gamma$ at a set $E$ with respect to the flow $\Phi$, respectively as
$$
{\frac{\mathrm{d}}{\mathrm{d}t}\np^\gamma(E_t)}_{|_{t=0}}\,,\quad\quad\quad\quad{\frac{\mathrm{d^2}}{\mathrm{d}t^2}\np^\gamma(E_t)}_{|_{t=0}}\,.
$$
\end{definition}

We recall here the result present in \cite[Theorem~3.1]{AFM} for the computation of the first and the second variation.

\begin{theorem}\label{thm:var}
Let $E$, $X$ and $\Phi$ as above. 
Then the first variation of $\np^\gamma$ computed at $E$ with respect to the flow $\Phi$ is given by
\begin{equation} \label{eq:pfirvar}
{\frac{\mathrm{d}}{\mathrm{d}t}\np^\gamma(E_t)}_{|_{t=0}} = \int_{\partial E} (H_{\partial E} + 4\gamma v^E)(X\cdot\nu_{E}) \,\dh\,,
\end{equation}
while the second variation of $\np^\gamma$ at $E$ with respect to the flow $\Phi$ reads as
\begin{align*}
&{\frac{\mathrm{d^2}}{\mathrm{d}t^2} \np^\gamma(E_t)}_{|_{t=0}}
= \int_{\partial E} \left( |D_{\tau}( X\cdot\nu_{E})|^2-|B_{\partial E}|^2(X\cdot\nu_{E})^2 \right) \,\dh \\
&\hspace{0.3 cm} + 8\gamma \int_{\partial E}\int_{\partial E}G_{\Pi^N}(x,y)(X(x)\cdot\nu_{E}(x))(X(y)\cdot\nu_{E}(y)) \,\dh(x)\dh(y)  \\
&\hspace{0.3 cm} +4\gamma\int_{\partial E} \partial_{\nu_{E}}v^E\,(X\cdot\nu_{E})^2\,\dh
-\int_{\partial E} (H_{\partial E}+4\gamma v^E)\,\div_{\tau}\bigl(X_{\tau}(X\cdot\nu_{E})\bigr) \,\dh\,. \\
\end{align*}
\end{theorem}

\begin{remark}
Notice that the last term of the second variation vanishes whenever $E$ is a critical set.
\end{remark}

We now follow the ideas contatined in \cite{AFM}. We introduce the space
$$
\widetilde{H}^1(\partial E):=\left\{\varphi\in H^1(\partial E) \::\: \int_{\partial E}\varphi\,\dh=0 \right\}\,,
$$
endowed with the norm $\|\varphi\|_{\widetilde{H}^1(\partial E)} := \|\nabla\varphi\|_{L^2(\partial E)}$.
On such a space we define the following quadratic form associated with the second variation.

\begin{definition} \label{def:fquad1}
Let $E\subset\Pi^N$ be a regular critical set.
We define the quadratic form $\partial^2\np^\gamma(E):\widetilde{H}^1(\partial E)\rightarrow\R$ by
\begin{equation} \label{eq:pfquad}
\begin{split}
\partial^2\np^\gamma(E)[\varphi]
& :=\int_{\partial E} \bigl( |D_{\tau}\varphi|^2-|B_{\partial E}|^2\varphi^2 \bigl) \,\dh
    +4\gamma\int_{\partial E} (\partial_{\nu_{E}}v^E) \varphi^2\,\dh \\
& \quad +8\gamma\int_{\partial E}\int_{\partial E}G_{\Pi^N}(x,y) \varphi(x) \varphi(y)\,\dh(x)\dh(y)\\
& =: \partial^2\per_{\Pi^N}(E)[\varphi]+\gamma\partial^2\nl_{\Pi^N}(E)[\varphi]\,,
\end{split}
\end{equation}
where $\partial^2\per_{\Pi^N}(E)$ denotes the first integral, while $\gamma\partial^2\nl_{\Pi^N}(E)$ the other two.
\end{definition}

Since our functional is translation invariant, if we compute the second variation of $\np^\gamma$ at a regular set $E$ with respect to a flow of the form $\Phi(x,t):=x+t\eta e_i$, where $\eta\in\R$ and $e_i$ is an element of the canonical basis of $\R^N$, setting $\nu_i:=\langle\nu_E,e_i\rangle$ we obtain that
$$
\partial^2 \np^\gamma(E)[\eta\nu_i]= {\frac{\mathrm{d}^2}{\mathrm{d} t^2}\np^\gamma(E_t)}_{|_{t=0}} = 0\,.
$$
Hence we need to avoid degenerate directions. Write
$$
\widetilde{H}^1(\partial E)=T^{\bot}(\partial E)\oplus T(\partial E)\,,
$$
where $T^{\bot}(\partial E)$ is the orthogonal complement to $T(\partial E)$ in the $L^2$-sense, \emph{i.e.},
\begin{equation}\label{eq:tbot}
T^{\bot}(\partial E):=\left\{  \varphi\in\widetilde{H}^1(\partial E) \,:\, \int_{\partial E}\varphi\nu_i\,\dh=0\;\;\mbox{for each}\;i=1,\dots,N \right\}\,.
\end{equation}
It can be shown (see \cite[Equation~(3.7)]{AFM}) that there exists an orthonormal frame $(\varepsilon_1,\dots,\varepsilon_N)$ such that
\begin{equation}\label{eq:diag}
\int_{\partial E}(\nu\cdot\varepsilon_i)(\nu\cdot\varepsilon_j)\,\dh=0\qquad\mbox{for all }i\neq j\,.
\end{equation}

\begin{definition} \label{def:varpos1}
We say that $\np^\gamma$ has \emph{strictly positive second variation} at the regular critical set $E$ if
$$
\partial^2\np^\gamma(E)[\varphi]>0\qquad\text{for all }\varphi\in T^\bot(\partial E)\setminus\{0\}.
$$
\end{definition}

We are now in position to recall the local minimality result of Acerbi, Fusco and Morini (see \cite[Theorem 1.1]{AFM}).

\begin{theorem}\label{thm:locmin}
Let $E\subset\Pi^N$ be a regular critical set such that $\np^\gamma$ has strictly positive second variation at $E$. Then there exist constants $C,\delta>0$, such that
$$
\np^\gamma(F)\geq\np^\gamma(E)+C\bigl(\alpha(E,F)\bigr)^2\,,
$$
whenever $F\subset\Pi^N$ with $|F|=|E|$ is such that $\alpha(E,F)\leq\delta$.
\end{theorem}


\section{Results}

\subsection{Minimality in small domains}

The first result we would like to prove is a local minimality property of critical points with respect to sufficiently small perturbations.

\begin{proposition}\label{prop:small}
Let $E\subset\Pi^N$ be a critical point for the functional $\np^{\gamma}$.
Then there exists $\varepsilon>0$ such that
$$
\np^{\gamma}(E)\leq\np^{\gamma}(F)\,,
$$
for any set $F\subset\Pi^N$ having $E\triangle F\Subset B_{\varepsilon}(x)$, for some $x\in \bar{E}$.
\end{proposition}

\begin{proof}[Sketch of the proof]
\emph{First part}. We first want to prove that we can find $\widetilde{\varepsilon}>0$ such that
\begin{equation*}
\np^\gamma(E)\leq\np^\gamma(F)\,,
\end{equation*}
whenever $F$ is a subset of $\Pi^N$ having $E\triangle F\Subset B_{\widetilde{\varepsilon}}(x)$, for some $x\in\partial E$.

Fix $\bar{x}\in\partial E$. The idea is to adapt to our case the proofs of the various steps leading to \cite[Theorem 1.1]{AFM}.\\
\emph{Step 1}. For any $\varepsilon>0$ sufficiently small, the following Poincar\'{e} inequality holds:
$$
\int_{\partial E\cap B_\varepsilon(\bar{x})} |D_\tau\varphi|^2 \dh\geq C_\varepsilon \int_{\partial E\cap B_\varepsilon(\bar{x})} \varphi^2 \dh\,,
$$
whenever $\varphi\in H^1(\partial E)$ has support contained in $B_\varepsilon(x)$. We know that $C_\varepsilon\rightarrow+\infty$ as
$\varepsilon\rightarrow0$. Let $M>0$ such that
$$
|B_{\partial E}|< M\,,\quad\quad\quad |\partial_\nu v^E|<M\,,
$$
and take $\varepsilon>0$ such that $C_{2\varepsilon}>M(1+4\gamma)$. Notice that it is possible to write
\begin{equation}\label{eq:write}
\int_{\partial E}\int_{\partial E}G_{\Pi^N}(x,y) \varphi(x) \varphi(y)\,\dh(x)\dh(y) = \int_{\Pi^N} |\nabla z|^2\dx\,, 
\end{equation}
where $-\triangle z=\varphi\h\measurerestr\partial E$. Thus, we have that
\begin{equation}\label{eq:posloc}
\partial^2\np^\gamma(E)[\varphi]>0\,,
\end{equation}
for any $\varphi\in H^1(\partial E)\meno\{0\}$ with support contained in $B_{2\varepsilon}(\bar{x})$.\\

\emph{Step 2}. We claim that it is possible to find constants $\delta>0$ and $C_0>0$ such that
$$
\np^{\gamma}(E) + C_0\bigl( \alpha(E,F) \bigr)^2 \leq \np^{\gamma}(F)\,,
$$
whenever $F\subset\Pi^N$, with $|F|=|E|$, is such that $\partial F=\{x+\psi(x)\nu_E(x)\;:\;x\in\partial E\}$, for some
$\|\psi\|_{W^{2,p}(\partial E)}\leq\delta$ with support contained in $B_{2\varepsilon}(\bar{x})$, for $p>\max\{2,N-1\}$.
We use the two step technique of \cite[Theorem 3.9]{AFM}. We first prove that we can find constants $\delta>0$ and $D>0$ such that
\begin{align*}
\inf\Bigl\{ \partial^2\np^\gamma(F)[\varphi]\;:\; &\varphi\in \widetilde{H}^1(\partial F)\,,\,\|\varphi\|_{H^1(\partial F)}=1\,,\,\\
&\mathrm{supp}(\varphi)\subset B_{2\varepsilon}(x)\,,\, \Bigl| \int_{\partial F} \varphi\nu_F\dh \Bigr|\leq \delta\,\Bigr\}\geq D\,,
\end{align*}
whenever $F\subset\Pi^N$, with $|F|=|E|$, is such that
$$
\partial F=\{ x+\psi(x)\nu_E(x) \,:\, x\in\partial E \}\,,
$$
for some $\psi\in W^{2,p}(\partial E)$ with $\|\psi\|_{W^{2,p}(\partial E)}\leq\delta$.
To prove it, we reason by contradiction as in the first step of the proof of \cite[Theorem 3.9]{AFM}.

Consider the flow $\Phi$, given by Lemma ~\ref{lem:X}, connecting the sets $E$ and $F$, and let $E_t:=\Phi_t(E)$. Then it is possible to write
$$
\np^\gamma(F)-\np^\gamma(E)= \int_0^1 (1-t)\Bigl( \partial^2\np(E_t)[X\cdot\nu_{E_t}] -
                     \int_{\partial E_t}(4\gamma v^{E_t}+H_t)\div_{\tau_t}(X_{\tau_t}(X\cdot \nu_{E_t})) \Bigr)\mathrm{d}t\,,
$$
where $\div_{\tau_t}$ is the tangential divergence on $\partial E_t$ and $X_{\tau_t}:=(X\cdot\tau_{E_t})\tau_{E_t}$.
It is possible to estimate from below of the integral, as it is done in the second step of the proof of \cite[Theorem 3.9]{AFM}.
Namely, it is possible to find $\delta>0$ such that
$$
\Bigl|  \int_{\partial E_t}(4\gamma v^{E_t}+H_t)\div_{\tau_t}(X_{\tau_t}(X\cdot \nu_{E_t})) \mathrm{d}t \Bigr|\leq
               \frac{D}{2}\|X\cdot\nu^{E_t}\|^2_{H^(\partial E_t)}\,,
$$
for all $t\in[0,1]$. Thus, with the above uniform coercivity property of $\partial^2\np(E_t)$ in force, we conclude.\\

\emph{Step 3}. For any $\varepsilon\ll 1$, let $\mathcal{I}_\varepsilon\subset B_{\sqrt{\varepsilon}}(\bar{x})$ be a smooth open set with the following properties:
the curvature of $\mathcal{I}_\varepsilon$ is uniformly bounded with respect to $\varepsilon$, the sets
$E\cup \mathcal{I}_\varepsilon$ and $E\meno \mathcal{I}_\varepsilon$ are smooth and $B_\varepsilon(\bar{x}) \subset \mathcal{I}_\varepsilon$
(see Figure ~\ref{fig:tn}). We claim that it is possible to find $\varepsilon>0$ such that
$$
\np^\gamma(E)\leq\np^\gamma(F)\,,
$$
for every set $F\subset\Pi^N$ with $|F|=|E|$, such that $E\triangle F\Subset \mathcal{I}_\varepsilon$.
The proof of such a result is similar to those of \cite[Theorem 4.3]{AFM}, where we reason by the sake of contradiction as follows:
suppose there exist a sequence $\varepsilon_n\rightarrow0$ and a corresponding sequence of sets $(F_n)_n$ with
$|F_n|=|E|$ and $E\meno \mathcal{I}_{\varepsilon_n}\subset F_n \subset E\cup\mathcal{I}_{\varepsilon_n}$, such that
$$
\np^\gamma(F_n)<\np^\gamma(E)\,.
$$
Using the uniform bound on the curvatures of the $\mathcal{I}_{\varepsilon_n}$'s, it is possible to prove, as in the first step of the proof of
\cite[Theorem 4.3]{AFM}, that we can find a sequence of uniform $(\omega,r)$-minimizers of the area functional $(E_n)_n$ with $|E_n|=|E|$ having $E_n\triangle E\Subset \mathcal{I}_{\varepsilon_n}$ and such that
$\np^\gamma(E_n)<\np^\gamma(E)$.
Thus, the improved convergence result stated in Theorem ~\ref{thm:W} allows us to say that the $E_n$'s converge to $E$ in the $C^{1,\beta}$-topology.
Finally, using the Euler-Lagrange equation satisfied by the $E_n$'s, it is also possible to prove that the
$E_n$'s actually converge to $E$ in the $W^{2,p}$-topology. This is in contradiction with the result of the previous step.\\

\begin{figure}
\includegraphics[scale=0.6]{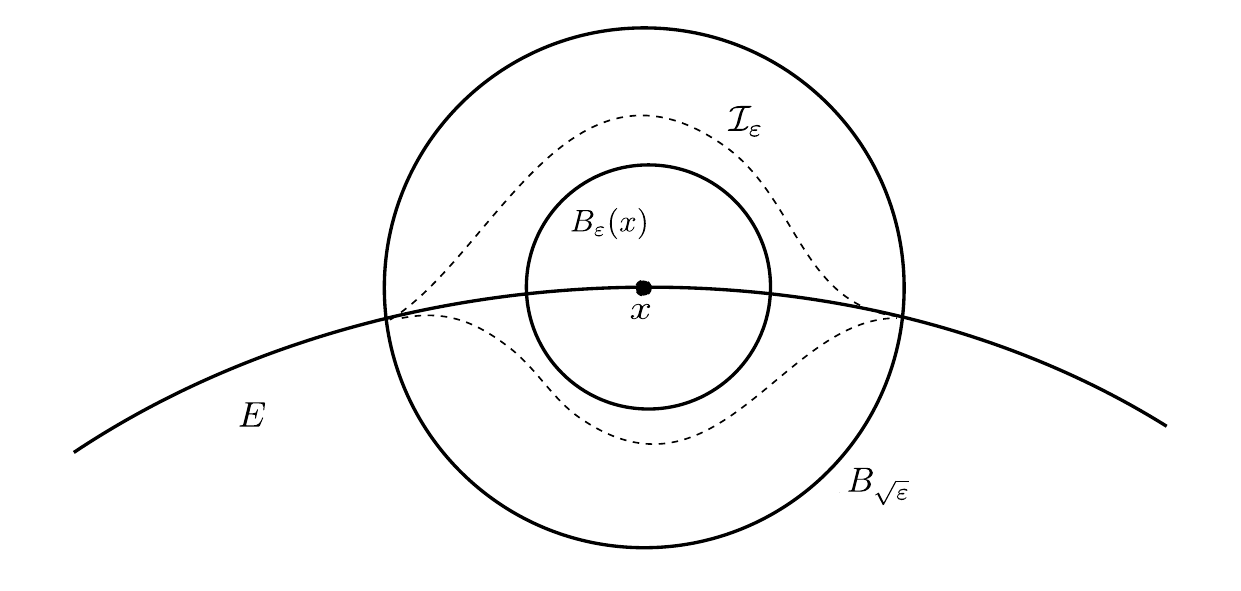}
\caption{An example of the set $\mathcal{I}_\varepsilon$.}
\label{fig:tn}
\end{figure} 

\emph{Step 4}. We now have to prove that the above constants can be made uniform with respect to $x\in\partial E$. Let us reason as follows: for any point
$x\in\partial E$, consider the ball $B_{\varepsilon(x)}(x)$, where $\varepsilon(x)>0$ is the radius found in Step 3 above.
Then it is possible to cover $\partial E$ with a finite family of such balls, let us say
$\left(B_{\varepsilon(x_i)}(x_i)\right)_{i=1}^L$.
We now claim that it is possible to find a constant $\widetilde{\varepsilon}>0$ with the following property: for any point $x\in\partial E$, there exists $i\in\{1,\dots,L\}$ such that $B_{\widetilde{\varepsilon}}(x)\subset B_{\varepsilon(x_i)}(x_i)$.
Indeed, let us suppose by contradiction that the claim is not true, \emph{i.e.}, there exist $(y_n)_n\subset\partial E$ such that $B_{\frac{1}{n}}(y_n)\not\subset B_{\varepsilon(x_i)}$ for all $i=1,\dots,L$. Then, by compactness of $\partial E$, up to a subsequence $y_n\rightarrow y\in\partial E$. By assumption $B_r(y)\not\subset B_{\varepsilon(x_i)}$ for all $i=1,\dots,L$. This contradicts the fact the covering property of the family of balls.

We can also suppose $\widetilde{\varepsilon}<\varepsilon(x_i)$ for each $i=1,\dots,L$.\\

\emph{Second part}. We now want to prove that we can find $\varepsilon\in(0,\widetilde{\varepsilon}/2)$ such that
\begin{equation}\label{eq:ineq}
\np^\gamma(E)\leq\np^\gamma(F)\,,
\end{equation}
whenever $F\subset\Pi^N$ is such that $E\triangle F\Subset B_{\varepsilon}(x)$, for some
$x\in E\meno(\partial E)_{\widetilde{\varepsilon}/2}$.

The key point is to observe the following:
\begin{equation}\label{eq:fineq}
|\nl(F)-\nl(E)|\leq c_0|E\triangle F|\leq C \per(E\triangle F)^{\frac{N}{N-1}}= C\bigl( \per(F)-\per(E) \bigr)^{\frac{N}{N-1}}\,.
\end{equation}
Indeed, the first inequality follows from the lipschitzianity of the nonlocal term (see Proposition \ref{prop:lip}), the second one from the quantitative isoperimetric inequality (see \cite{FusMagPra}), and the last one from the fact that $E\triangle F\Subset B_{\varepsilon}(x)$, with $x$ in the interior of $E$.
Notice that \eqref{eq:ineq} can be written as
$$
\per(F)-\per(E)\geq \gamma\bigl( \nl(E)-\nl(F) \bigr)\,.
$$
Using \eqref{eq:fineq} and the fact that $t^{\frac{N}{N-1}}<Ct$ for $t$ small, we know that the above inequality is satisfied if $\per(F)-\per(E)<\delta$, for some $\delta>0$. If instead it holds $\per(F)-\per(E)\geq\delta$, we obtain the validity of \eqref{eq:ineq} by noticing that
$$
|\nl(F)-\nl(E)|\leq c_0|E\triangle F|\leq C\varepsilon^N\,,
$$
and by taking $\varepsilon$ sufficiently small. This concludes the proof.
\end{proof}


\subsection{Uniform local minimizers}

We start by proving a lemma that will be used several times. The proof can be found in \cite{AFM} (Step 4 of the proof of Theorem 3.4), but we prefer to report it here for reader's convenience.

\begin{lemma}\label{lem:convc3}
Let $E\subset\Pi^N$ be a critical set for $\np^{\bar{\gamma}}$, with $\bar{\gamma}\geq0$. Then for any $\varepsilon>0$ it is possible to find
$\widetilde{\varepsilon}>0$ with the following property:
if $E_\gamma$ is a critical point of $\np^\gamma$, with $\gamma\in(\bar{\gamma}-\varepsilon,\bar{\gamma}+\varepsilon)$ such that
$\mathrm{d}_{C^1}(E,E_\gamma)<\varepsilon$, then $\mathrm{d}_{C^{3,\beta}}(E,E_\gamma)<\widetilde{\varepsilon}$, for all $\beta\in(0,1)$. Moreover, $\widetilde{\varepsilon}\to0$ as $\varepsilon\to0$.
\end{lemma}

\begin{proof}
Take a sequence $\gamma_n\rightarrow\bar{\gamma}$ and a sequence $(E_{\gamma_n})_n$, where $E_{\gamma_n}$ is a critical point of
$\np^{\gamma_n}$, with $\mathrm{d}_{C^1}(E,E_{\gamma_n})\rightarrow0$.
Thanks to the $C^1$-convergence of $E_{\gamma_n}$ to $E$ and by standard elliptic estimates, it is easy to see that
\begin{equation}\label{e3}
v^{E_{\gamma_n}}\rightarrow v^E \qquad\text{in }C^{1,\beta}(\Pi^N)\,,
\end{equation}
for all $\beta\in(0,1)$. 
Since we have
\begin{equation}\label{e1}
H_{\partial E} = \lambda- 4\bar{\gamma} v^{E}
\end{equation}
in force on $\partial E$, for some constant $\lambda$, and
\begin{equation}\label{e2}
H_{\partial E_{\gamma_n}} = \lambda_{\gamma_n} - 4{\gamma_n} v^{E_{\gamma_n}}
\end{equation}
valid $\partial E_{\gamma_n}$, if we prove that $\lambda_{\gamma_n}\rightarrow\lambda$, by standard elliptic estimates, we get that $E_{\gamma_n}\rightarrow E$ in $C^{3,\beta}$.
We work locally, by considering a cylinder $C=B'\times(-L,L)$, where $B'\subset\R^{N-1}$ is a ball centered at the origin,
such that in a suitable coordinate system we have
\begin{align*}
E_{\gamma_n}\cap C   &= \{ (x',x_N)\in C : x'\in B',\, x_N<g_{\gamma_n}(x') \},\\
E\cap C     &= \{ (x',x_N)\in C : x'\in B',\, x_N<g(x') \}
\end{align*}
for some functions $g_{\gamma_n}\to g$ in $C^{1,\beta}(\overline{B'})$.
By integrating \eqref{e2} on $B'$ we obtain
\begin{align*}
\lambda_{\gamma_n}&\mathcal{H}^{N-1}(B') - 4{\gamma_n}\int_{B'} v^{E_{\gamma_n}}(x',g_{\gamma_n}(x'))\,\dh(x')\\
&   = - \int_{B'} \div \biggl( \frac{\nabla g_{\gamma_n}}{\sqrt{1+|\nabla g_{\gamma_n}|^2}} \biggr) \,\dh(x')
    = -\int_{\partial B'} \frac{\nabla g_{\gamma_n}}{\sqrt{1+|\nabla g_{\gamma_n}|^2}} \cdot \frac{x'}{|x'|} \,\mathrm{d}\mathcal{H}^{N-2}\,,
\end{align*}
and the last integral in the previous expression converges, as $n\rightarrow\infty$, to
\begin{align*}
&-\int_{\partial B'} \frac{\nabla g}{\sqrt{1+|\nabla g|^2}} \cdot \frac{x'}{|x'|} \,\mathrm{d}\mathcal{H}^{N-2}
    = - \int_{B'} \div \biggl( \frac{\nabla g}{\sqrt{1+|\nabla g|^2}} \biggr) \,\dh(x')  \\
&= \lambda\mathcal{H}^{N-1}(B') - 4\bar{{\gamma_n}}\int_{B'} v^{E_{\bar{{\gamma_n}}}}(x',g_{\gamma_n}(x'))\,\dh(x') \,,
\end{align*}
where the last equality follows by \eqref{e1}. This shows, recalling \eqref{e3}, that $\lambda_{\gamma_n}\to\lambda$,
as $n\rightarrow\infty$.
\end{proof}

We now state the main result of this section, namely a uniform local minimality result for strictly stable critical points of $\np^\gamma$.

\begin{proposition}\label{prop:unif}
Let $E\subset\Pi^N$ be a strictly stable critical point for $\np^{\bar{\gamma}}$, $\bar{\gamma}\geq0$. Then there exist constants $\delta>0$, $\varepsilon>0$, $\widetilde{\gamma}>0$ and $C>0$ with the following property:
take $\gamma\in(\bar{\gamma}-\widetilde{\gamma},\bar{\gamma}+\widetilde{\gamma})$ and
let $E_\gamma$ be a critical point for $\np^\gamma$ with $\mathrm{d}_{C^1}(E,E_\gamma)<\varepsilon$; then
$$
\np^\gamma(E_\gamma)+C\bigl( \alpha(E_\gamma,F) \bigr)^2\leq\np^\gamma(F)\,,
$$
for every set $F\subset\Pi^N$, with $|F|=|E_\gamma|$, such that $\alpha(E_\gamma,F)\leq\delta$.
\end{proposition}

The proof of Proposition \ref{prop:unif} will follow the same strategy performed in \cite{AFM}.
The difficulty here is to check that all the estimates provided there can be made uniform with respect to the $C^1$ closeness of $E_\gamma$ to $E$.
We recall that, in checking this, we in fact simplify the general argument, by replacing \cite[Lemma 3.8]{AFM} with a penalization argument, that was inspired to us by \cite{DamLam}.
For reader's convenience, we recall here the general strategy we are going to use to prove Proposition \ref{prop:unif}.
It is divided into two main steps: the first one is to prove a quantitative isolated local minimality property of a strictly stable critical point $E$ with respect to the $W^{2,p}$-topology, \emph{i.e.}, with respect to sets whose boundary is a $W^{2,p}$ graph on $\partial E$ (see Lemma \ref{lem:minw2p}). For technical reasons, we will first prove this kind of local minimality property for a penalized functional (see Lemma \ref{lem:w2p}) and then prove that this implies the required local minimality property.
Then, we are going to use a selection principle argument (see \cite{AFM} and \cite{CicLeo}) to obtain the $L^1$ isolated local minimality from the $W^{2,p}$ one. This will also require an intermediate technical result stating an $L^\infty$ local minimality property for $E$.

Let us now start by proving a uniform version of a technical result present in \cite{AFM}.

\begin{definition}\label{def:tn}
Let $F\subset\Pi^N$ be a set of class $C^\infty$. We will denote by $\mathcal{N}_\mu(F)$, with $\mu>0$, a tubular neighborhood of $F$ where the signed distance $\mathrm{d}_F$ from $F$ and the projection $\pi_F$ on $\partial F$ are smooth in $\mathcal{N}_\mu(F)$.
\end{definition}

\begin{lemma}\label{lem:X}
Let $E\subset\Pi^N$ be a strictly stable critical point for $\np^{\bar{\gamma}}$, $\bar{\gamma}\geq0$, and let $p>\max\{2,N -1\}$.
Then, there exist constants $\mu>0$, $\widetilde{\gamma}>0$, $\varepsilon>0$ and $C>0$ with the following property:

for any critical point $E_\gamma$ of $\np^\gamma$, with
$\gamma\in(\bar{\gamma}-\widetilde{\gamma},\bar{\gamma}+\widetilde{\gamma})$ and $\mathrm{d}_{C^1}(E,E_\gamma)<\varepsilon$, and any $\psi\in C^\infty(E_\gamma)$ with
$\|\psi\|_{W^{2,p}(\partial E_\gamma)}\leq\varepsilon$, there exists a vector field $X\in C^\infty$ with $\div X=0$ in $\mathcal{N}_\mu(E_\gamma)$ such that, if we consider its flow, \emph{i.e.}, the solution of
\begin{equation}\label{eq:flow}
\frac{\partial\Phi}{\partial t}= X(\Phi)\,,\quad\quad\Phi(0,x)=x\,,
\end{equation}
we have $\Phi(1,x)=x+\psi(x)\nu_{E_\gamma}(x)$, for any $x\in\partial E_\gamma$. Moreover, the following estimate holds true
$$
\|\Phi(t,\cdot)-\id\|_{W^{2,p}(\partial E_\gamma)}\leq C \|\psi\|_{W^{2,p}(\partial E_\gamma)}\,.
$$
Finally, set $E^t_\gamma:=\Phi(t,E_\gamma)$, and suppose $|E^1_\gamma|=|E_\gamma|$. Then $|E_\gamma^t|=|E_\gamma|$ for all $t\in[0,1]$, and
$$
\int_{\partial E_\gamma^t} X\cdot\nu_{E_\gamma^t} \dh = 0\,.
$$
\end{lemma}

\begin{proof}
Take $0<\varepsilon<\varepsilon_0$, where $\varepsilon_0>0$ is the constant given by Lemma ~\ref{lem:convc3}. Then, possibly reducing $\varepsilon$, we can find $\mu>0$ and $\widetilde{\gamma}\in(0,\varepsilon)$ such that $\mathcal{N}_\mu(E_\gamma)$ is a tubular neighborhood of $E_\gamma$ (see Definition ~\ref{def:tn}) for every $E_\gamma$ critical point of $\np^\gamma$, with $\gamma\in(\bar{\gamma}-\widetilde{\gamma},\bar{\gamma}+\widetilde{\gamma})$ and $\mathrm{d}_{C^1}(E,E_\gamma)<\varepsilon$.
For reader's convenience we recall the idea behind the construction of the vector field $X$.
Since we have to connect $E_\gamma$ with $F$, whose boundary is a normal graph over the boundary of $E_\gamma$, the simplest thing to do is to use a vector field that, in a neighborhood of $\partial E_\gamma$, coincides with the (usual extension) of the normal vector field $\nabla \mathrm{d}_{E_\gamma}$ on $\partial E_\gamma$. Since we have to satisfy a divergence-free condition, we have to modify the normal vector field $\nabla\mathrm{d}_{E_\gamma}$ as follows
\[
\widetilde{X}(z):=\xi(z)\nabla\mathrm{d}_{E_\gamma}\,,
\]
for $z\in\mathcal{N}_\mu(E_\gamma)$, in such a way that $\div\widetilde{X}=0$. This leads to the equation
\begin{equation}\label{eq:x}
\xi\triangle\mathrm{d}_{E_\gamma}+\nabla \xi\cdot\nabla\mathrm{d}_{E_\gamma}=0\,,\quad\quad\text{in } \mathcal{N}_\mu(E_\gamma)\,.
\end{equation}
Since, for $\delta>0$ sufficiently small, each point $z\in\mathcal{N}_\mu(E_\gamma)$ can be written as
$z=x+\mathrm{d}_{E_\gamma}(z)\nabla\mathrm{d}_{E_\gamma}(z)$, by defining the function
$f(x,t):=f_x(t):=\xi(x+t\mathrm{d}_{E_\gamma}(x))$, for $x\in\partial E_\gamma$ and $t\in(-\mu,\mu)$ for some small $\mu>0$, we can write \eqref{eq:x} as
\[
\left\{
\begin{array}{l}
(f_x)'(t)+f_x(t)\triangle \mathrm{d}_{E_\gamma}(x+t\nu_{E_\gamma}(x))=0\,,\\
f_x(0)=1\,.
\end{array}
\right.
\]
Thus, by taking
$$
\xi(x+t\nu_{E_\gamma}(x)):=f_x(t)=\exp\Bigl( -\int_0^t \triangle \mathrm{d}_{E_\gamma}(x+s\nu_{E_\gamma}(x))\mathrm{d}s \Bigr)\,,
$$
we are sure that $\div\widetilde{X}=0$ in $\mathcal{N}_\mu(E_\gamma)$. Finally, we have to ensure that at time $t=1$ the flow starting from $E_\gamma$ will be such that $\Phi(1,x)=x+\psi(x)\nu_{E_\gamma}(x)$. 
For, we modify our vector field $\widetilde{X}$ as follows
$$
X(z):=\frac{\psi\bigl( \pi_{E_\gamma}(z) \bigr)}{\int_0^1 \xi\bigl( \pi_{E_\gamma}(z) +s\nabla\mathrm{d}_{E_\gamma}\bigl(\pi_{E_\gamma}(z) \bigr) \bigr)\mathrm{d} s}\xi(z)\nabla d_{E_\gamma}(z)\quad\quad \text{ for } z\in\mathcal{N}_\mu(E_\gamma)\,,
$$
and extending it in a smooth way in the whole $\Pi^N$.
The relation between the constants $\varepsilon$ and $\mu$ is the following:
using again the $C^{3,\beta}$-closeness of $E_\gamma$ to $E$, it is possible to find a constant $C>0$ such that
$\|\psi\|_{L^\infty(\partial {E_\gamma})}\leq C \|\psi\|_{W^{2,p}(\partial {E_\gamma})}< C \varepsilon$ for any set $E_\gamma$ as above. Take $0<\varepsilon<\mu/C$.

We now prove some estimates on $\Phi$. First of all notice that we can find a constant $C>0$ such that, for every set $E_\gamma$ as above, it holds
$$
\|X\|_{W^{2,p}(\mathcal{N}_\mu(E_\gamma))}\leq C \|\psi\|_{W^{2,p}(\partial E_\gamma)}\,.
$$
Thus, by the definition of the flow $\Phi$, we have that
$$
\|\Phi-\id\|_{C^0(\mathcal{N}_\mu(E_\gamma))}\leq C \|\psi\|_{W^{2,p}(\partial E_\gamma)}\,.
$$
To estimate the other norms, we just differentiate in \eqref{eq:flow} to obtain
$$
\|\nabla_x\Phi(t,\cdot)-\id\|_{C^0(\mathcal{N}_\mu(E_\gamma))}\leq C_\mu\|\nabla X\|_{C^0(\mathcal{N}_\mu(E_\gamma))}\leq
            C_\mu \|\psi\|_{W^{2,p}(\partial E_\gamma)}\,.
$$
Since this shows that the $(N-1)$-dimensional Jacobian of $\Phi(t,\cdot)$ is uniformly close to $1$ on $\partial E_\gamma$, differentiating again in \eqref{eq:flow}, we obtain also the following estimate:
$$
\|\nabla^2_x\Phi(t,\cdot)\|_{L^p(\partial E_\gamma)}\leq C_\mu \|\nabla^2 X\|_{L^p(\mathcal{N}_\mu(E_\gamma))}\,.
$$
Finally, if $|E_\gamma^1|=|E_\gamma|$, then
$$
\frac{\mathrm{d}^2}{\mathrm{d} t^2}|E_t|=\int_{E_\gamma^t}(\div X)(X\cdot\nu_{E_\gamma})\dh=0\quad\quad\text{ for all } t\in[0,1]\,.
$$
This follows from \cite[Equation (2.30)]{ChoSte}. Thus, the function $t\mapsto|E_\gamma^t|$ is affine in $[0,1]$, and since $|E_\gamma|=|E_\gamma^t|$,
we have that it is constant. So
$$
0=\frac{\mathrm{d}}{\mathrm{d} t}|E_t|=\int_{E_\gamma^t}\div X\dh=\int_{\partial E_\gamma^t} X\cdot\nu_{E_\gamma^t}\dh\,.
$$
This concludes the proof of the lemma.
\end{proof}

We now introduce the penalization we are going to use as an intermediate step in the proof of the uniform $W^{2,p}$-local minimality. The aim is to get rid of the technicality introduced in \cite{AFM} in order to deal with the translation invariance of the functional, for which we would have to check that the estimates they provide are uniform with respect to the parameter $\gamma$. By looking at the steps of the proof of \cite[Theorem 3.9]{AFM}, it turns out that a way to get rid of the above technicalities is the following: given a critical and strictly stable set $E$, we need to add to $\np^\gamma$ a penalization, depending on $E$, that vanishes at $E$ and whose second variation, computed at a set $F$, is strictly positive for functions $\varphi\in H^1(\partial F)\meno T^\perp(\partial F)$.
In order to define our penalized functional we need the following technical lemma, whose proof is left to the reader.

\begin{lemma}\label{lem:f}
Let $E\subset\Pi^N$ be a regular set, and let $M>\|\nu_E\|_{C^1(\partial E)}$.
Then there exist a function $f_E:\Pi^N\rightarrow\R^N$ such that:
\begin{itemize}
\item[(i)]  $f_E(z+t\nu_E(z))=\nu_E(z)$ for every point $z+t\nu_E(z)\in\R^N$, where $z\in\partial E$ and $|t|<t_0$, for some $t_0>0$ depending on $E$,
\item[(ii)] $\|f_E\|_{C^1(\Pi^N;\R^N)}<M$.
\end{itemize}
Moreover, it is possible to find $\varepsilon>0$ with the following property.
For every $F\subset\Pi^N$ with $\mathrm{d}_{C^1}(E,F)<\varepsilon$, and every $\eta>0$,
it is possible to find $\widetilde{\eta}>0$ such that
\begin{equation}\label{eq:imp}
\Bigl| \int_{\partial F^\psi} \varphi f_E\dh \Bigr|\leq \eta\quad \Rightarrow
         \quad \Bigl| \int_{\partial F^\psi} \varphi\nu_{F^\psi}\dh \Bigr|\leq \widetilde{\eta}\,,
\end{equation}
for any function $\varphi\in \widetilde{H}^1(\partial F^\psi)$ with $\|\varphi\|_{H^1(\partial F^\psi)}=1$,
whenever $F^\psi\subset\Pi^N$ is such that $\partial F^\psi=\{ x+\psi(x)\nu_F(x) \,:\, x\in \partial F \}$ for some
$\|\psi\|_{W^{2,p}(\partial F)}\leq\eta$. In particular, it holds that $\widetilde{\eta}\to0$ as $\eta\to0$.
\end{lemma}

We are now in position to define our penalized functional.

\begin{definition}\label{def:fpen}
Let $E\subset\Pi^N$ be regular set, and let $F\subset\Pi^N$. We define the penalized functional $\np^\gamma_E:\Pi^N\to[0,\infty)$ as
$$
\np^\gamma_{f_E,F}(G):=\np^\gamma(G) + \mathrm{Pen}_{f_E,F}(G)\,,
$$
where the penalization $\mathrm{Pen}_E:\Pi^N\to[0,\infty)$ is defined via
$$
\mathrm{Pen}_{f_E,F}(G):=\Bigl| \int_G f_E(x)\,\mathrm{d}x - \int_F f_E(x)\,\mathrm{d}x \Bigr|^2\,.
$$
and $f_E:\Pi^N\to\R^N$ is the function given by Lemma \ref{lem:f}, relative to $E$.
\end{definition}

\begin{remark}
Notice that $\mathrm{Pen}_{f,E}(E)=0$, for every function $f:\Pi^N\to\R$ and every set $E\subset\Pi^N$.
\end{remark}

In the following lemma we calculate the first and the second variation of the penalization $\mathrm{Pen}_E$.

\begin{lemma}
Let $E,F\subset\Pi^N$ and let $f:\Pi^N\to\R$ be a $C^1$ function.
Consider an admissible family of diffeomorphisms $(\Phi(\cdot,t))_t$ for $F$ (see Definition \ref{def:adm}).
Then the first variations of $\mathrm{Pen}_{f,E}$ computed at $F$ with respect to the family $(\Phi(\cdot,t))_t$ reads as
$$
{\frac{\mathrm{d}}{\mathrm{d} s}\mathrm{Pen}_{f,E}(F_s)}_{|s=t} 
    = 2 \Bigl( \int_{F_t} f\,\mathrm{d}x - \int_E f\,\mathrm{d}x \Bigr)\cdot
           \int_{\partial F_t} f_E(X\cdot\nu_{F_t})\dh\,,
$$
and the second variation is given by
\begin{align*}
&{\frac{\mathrm{d}^2}{\mathrm{d} t^2} \mathrm{Pen}_{f,E}(F_t)}_{|t=0} = 2 \,\Bigl| \int_{\partial F} f(X\cdot\nu_F)\dh \Bigr|^2 \\
&\quad\quad+ 2 \Bigl( \int_F f\,\mathrm{d}x - \int_E f\,\mathrm{d}x \Bigr)\cdot \int_{\partial F}f[(X\cdot\nu_F)\div X - \div_\tau(X_\tau(X\cdot\nu_F))]\dh\,.\\
\end{align*}
\end{lemma}

\begin{proof}
Fix $i=1,\dots,N$ and consider the scalar function $g:(-1,1)\rightarrow\R$ given by
$$
g(t):=\int_{F_t} f_i(x)\,\mathrm{d}x \,.
$$
Then
$$
g'(t)= \int_{F_t} \bigl( \nabla f_i\cdot X_t + f_i\div X_t \bigr)\,\mathrm{d}x 
    = \int_{\partial F_t} f_i(X_t\cdot\nu_{F_t})\dh\,.
$$
Moreover
\begin{align*}
g''(0)&= \frac{\mathrm{d}}{\mathrm{d}t} \Bigl( \int_{\partial F_t} f_i(X\cdot\nu_{F_t})\dh \Bigr)_{|_{t=0}}\\
&= \int_{\partial F} f_i\frac{\mathrm{d}}{\mathrm{d}t}\bigl( (X\circ\Phi_t)\cdot(\nu_{F_s}\circ\Phi_t)J^{N-1}\Phi_t \bigr)_{|_{t=0}}\dh \\
&\hspace{0.6cm} +\int_{\partial F} (\nabla f_i\cdot X)(X\cdot\nu_F)\dh \\
&= \int_{\partial F} f_i\bigl[ \div_\tau(X(X\cdot\nu_F)) + Z\cdot\nu-2X_\tau\cdot\nabla_\tau(X\cdot\nu) + D\nu_F[X_\tau, X_\tau] \bigr]\dh\\
&\hspace{0.6cm}+ \int_{\partial F} (\nabla f_i\cdot X)(X\cdot\nu_F)\dh  \\     
&= \int_{\partial F} f_i\bigl[ (X\cdot\nu)\div X - \div_\tau(X\tau(X\cdot\nu)) \bigr]\dh\,,
\end{align*}
where in the last step we used the same computations as in \cite[Theorem 3.1]{AFM}.
\end{proof}

Motivated by the fact that
\begin{equation}\label{eq:secbarpenf}
{\frac{\mathrm{d}^2}{\mathrm{d} t^2} \mathrm{Pen}_{f_E,E}(E_t)}_{|t=0} = 2\, \Bigl| \int_{\partial E} \nu_E(X\cdot\nu_{E})\dh \Bigr|^2\,,
\end{equation}
we introduce the following quadratic form.

\begin{definition}
For $\gamma\geq0$ and $f:\Pi^N\to\R$, we define the quadratic form
$\partial^2\np^\gamma_f(E):\widetilde{H}^1(\partial F)\rightarrow\R$ as:
$$
\partial^2\np^\gamma_f(E)[\varphi] := \partial^2\np^\gamma(E)[\varphi] + 2\, \Bigl| \int_{\partial E} \varphi f \dh \Bigr|^2\,.
$$
\end{definition}

\begin{remark}
Let $E\subset\Pi^N$ be be a strictly stable critical point for $\np^{\gamma}$. Then
$$
\partial^2\np^\gamma_{f_E}(E)[\varphi] > 0\quad\quad\text{ for all } \varphi\in \widetilde{H}^1(\partial E)\meno\{0\}\,. 
$$
Indeed, from \eqref{eq:secbarpenf}, the term due to the second variation of the penalization is non-negative and vanishes only for
$\varphi\in T^\bot(\partial E)$. Moreover, by the strict stability of $E$, we know that $\partial^2\np^{\gamma}(E)$ is strictly positive on $T^\bot(\partial F)\meno\{0\}$.
\end{remark}

The idea is to prove a $W^{2,p}$ local minimality result for the penalized functional and then check that it is possible to remove $\mathrm{Pen}_E$ and to obtain a similar result for the functional $\np^\gamma$.

\begin{lemma}\label{lem:w2p}
Let $p>\max\{2,N-1\}$, and let $E\subset\Pi^N$ be a strictly stable critical point for $\np^{\bar{\gamma}}$. Then there exist constants $\widetilde{\gamma}>0$, $\delta>0$, $\varepsilon>0$ and $C>0$ with the following property.

For $\gamma\in(\bar{\gamma}-\widetilde{\gamma},\bar{\gamma}+\widetilde{\gamma})$,
let $E_\gamma$ be a critical point for $\np^\gamma$ with $\mathrm{d}_{C^1}(E,E_\gamma)<\varepsilon$; then it holds that
$$
\np^\gamma_{f_E,E_\gamma}(F)\geq \np^\gamma_{f_E,E_\gamma}(E_\gamma) + C |E_\gamma\triangle F|^2\,,
$$
for every set $F\subset\Pi^N$ with $|F|=|E_\gamma|$ and $\partial F=\{ x+\psi(x)\nu_{E_\gamma}(x) \;:\; x\in\partial E_\gamma \}$ for some $\|\psi\|_{W^{2,p}(\partial E_\gamma)}\leq\delta$.
\end{lemma}

\begin{proof}
\emph{Step 1}. We claim that is possible to find constants $\widetilde{\gamma}>0$, $\delta>0$, $\varepsilon>0$ and $D>0$ such that, for any $\gamma\in(\bar{\gamma}-\widetilde{\gamma}, \bar{\gamma}+\widetilde{\gamma})$, any critical set $E_\gamma\subset\Pi^N$ for $\np^\gamma$, with $|E_\gamma|=|E|$ and $\mathrm{d}_{C^1}(E,E_\gamma)<\varepsilon$,
we have
\begin{equation}\label{eq:m}
\inf\Bigl\{ \partial^2\np^\gamma_{f_E,E_\gamma}(F)[\varphi]\;:\;\varphi\in \widetilde{H}^1(\partial F)\,,\,
    \|\varphi\|_{H^1(\partial F)}=1\,\Bigr\}\geq D\,,
\end{equation}
whenever $F\subset\Pi^N$, with $|F|=|E|$, is such that
$$
\partial F=\{ x+\psi(x)\nu_{E_\gamma}(x) \,:\, x\in\partial E_\gamma \}\,,
$$
for some $\psi\in W^{2,p}(\partial E_\gamma)$ with $\|\psi\|_{W^{2,p}(\partial E_\gamma)}\leq\delta$.

\emph{Part 1.} We first prove that we can find constants as above such that
$$
\inf\Bigl\{ \partial^2\np^\gamma_{f_E,E_\gamma}(F)[\varphi]\;:\;\varphi\in \widetilde{H}^1(\partial F)\,,\,
\|\varphi\|_{H^1(\partial F)}=1\,,\, \Bigl| \int_{\partial F} \varphi\nu_F\dh \Bigr| < \delta \Bigr\}\geq D\,,
$$
for sets $F\subset\Pi^N$ as above.

In this case we can reason as follows: suppose for the sake of contradiction that there exist a sequence $\gamma_n\rightarrow\bar{\gamma}$, a sequence of sets
$(E_{\gamma_n})_n$ with $|E_{\gamma_n}|=|E|$ and $E_{\gamma_n}\rightarrow E$ in $C^1$
(by Lemma ~\ref{lem:convc3} we can say that the convergence holds in $C^{3,\beta}$),
a sequence of sets $(F_n)_n$ with $|F_n|=|E|$ and
$$
\partial F_n=\{ x+\psi_n(x)\nu_{E_{\gamma_n}}(x) \,:\, x\in\partial E_{\gamma_n} \}\,,
$$
for $\psi_n\in W^{2,p}(\partial E_{\gamma_n})$ with $\|\psi_n\|_{W^{2,p}(\partial E_{\gamma_n})}\leq 1/n$, and a sequence of functions
$\varphi_n\in\widetilde{H}^1(\partial F_n)$ with $\|\varphi_n\|_{H^1(\partial F_n)}=1$ and
$\int_{\partial F_n}\varphi_n\nu_{F_n}\rightarrow0$, such that
$$
\partial^2\np^{\gamma_n}(F_n)[\varphi_n]\rightarrow0\quad\quad\text{as } n\rightarrow\infty\,.
$$
One can see that $E_{\gamma_n}\rightarrow E$ in $C^{3,\beta}$ implies that $F_n\rightarrow E$ in $W^{2,p}$.
Then there exist diffeomorphisms $\Phi_n: E\rightarrow F_n$ converging to the identity in $W^{2,p}(\partial E)$. The idea now is to consider the functions
$\widetilde{\varphi}_n\in\widetilde{H}^1(\partial E)$ defined as
$$
\widetilde{\varphi}_n:=\varphi_n\circ\Phi_n-a_n\,,
$$
where , $a_n:=\int_{\partial E} \varphi_n\circ\Phi_n\dh$, and to prove that
\begin{equation}\label{eq:conv1}
\partial^2\np^{\gamma_n}(F_n)[\varphi_n] - \partial^2\np^{\gamma_n}(E)[\widetilde{\varphi}_n]\rightarrow0\,,
\end{equation}
and that
\begin{equation}\label{eq:conv2}
\partial^2\np^{\gamma_n}(E)[\bigl(\widetilde{\varphi}_n\bigr)^\bot]-\partial^2\np^{\gamma_n}(E)[\widetilde{\varphi}_n]\rightarrow0\,,
\end{equation}
where $(\widetilde{\varphi}_n\bigr)^\bot$ is the $L^2$-orthogonal projection of $\widetilde{\varphi}_n$ on $T^\bot{\partial}$ (see \eqref{eq:tbot}).
The above convergences are proved exactly as in Step 1 of \cite[Theorem 3.9]{AFM}, where we notice that the convergence of the term of the quadratic form due to the penalization, is easily seen to converge.

This allows to conclude: indeed, from the fact that
\begin{equation}\label{eq:conv3}
\partial^2\np^{\gamma_n}(E)[\bigl(\widetilde{\varphi}_n\bigr)^\bot]-\partial^2\np^{\bar{\gamma}}(E)[\bigl(\widetilde{\varphi}_n\bigr)^\bot]\rightarrow0\,,
\end{equation}
we obtain a contradiction with
$$
\inf\bigl\{ \partial^2\np^{\bar{\gamma}}(E)[\varphi]\;:\; \varphi\in T^\bot(\partial E)\setminus\{0\},\, 
                \|\varphi\|_{H^1(\partial E)}=1 \bigr\}\geq C>0\,.
$$
This last fact follows from the strict positivity of the second variation (see \cite[Lemma 3.6]{AFM}).
In order to prove \eqref{eq:conv1} and \eqref{eq:conv2} we have just to repeat the same computation as in step 1 of \cite[Theorem 3.9]{AFM}.
Finally \eqref{eq:conv3} is easily seen to be true.\\

\emph{Part 2.} Let $\eta>0$ be the constant provided by Lemma \ref{lem:f}. Assume that the corresponding $\widetilde{\eta}>0$ is such that  $\widetilde{\eta}<\delta$, where $\delta>0$ is the constant we found in Part 1.
Then we have two possibilities. In the case where
$$
\Bigl| \int_{\partial F} \varphi f_E \,\dh \Bigr| > \eta\,,
$$
the claim follows by using the definition of $\partial^2\np^\gamma_{E_\gamma}(F)$, from which we get that $\partial^2\np^\gamma_{E_\gamma}(F)[\varphi] > 2\eta^2$.
Otherwise, the following inequality is in force
\begin{equation}\label{eq:case}
\Bigl| \int_{\partial F} \varphi f_E\dh \Bigr| \leq \eta\,,
\end{equation}
and in this case by Lemma \ref{lem:f} we infer that
\[
\Bigl| \int_{\partial F} \varphi \nu_F\dh \Bigr| \leq \widetilde{\eta}<\delta\,.
\]
So that the validity of the claim is provided by the result proved in the previous part.\\

\emph{Step 2}. To conclude, we have to check that all the estimates needed in the second step of \cite[Theorem 3.9]{AFM} can be made uniform with respect to
$\gamma\in(\bar{\gamma}-\widetilde{\gamma}, \bar{\gamma}+\widetilde{\gamma})$.
For any pair of sets $E_\gamma$ and $F$ as in the statement, consider the vector field $X_\gamma$ and its flow $\Phi_\gamma(\cdot,t)$, provided by Lemma ~\ref{lem:X}. Let $E_\gamma^t:=\Phi_\gamma(E_\gamma,t)$. Fixed $\varepsilon>0$, it is possible to find $\varepsilon>0$ and $\delta>0$ such that
$$
\| \nu_{E_\gamma} - \nu_{E_\gamma^t}\bigl(\Phi_n(\cdot,t)\bigr) \|_{L^\infty}<\varepsilon\,,\quad\quad
           \| J^{N-1}\bigl(\Phi_\gamma(\cdot,t)\bigr) - 1 \|_{L^\infty}<\varepsilon\,.
$$
Moreover, thanks to the $C^1$-closeness of $E_\gamma^t$ to $E$, we can also suppose
$$
\|4\gamma v^{E_\gamma^t} + H_{E^t_\gamma} - \lambda_\gamma  \|_{L^\infty}<\varepsilon\,,
$$
where $4\gamma v^{E_\gamma} + H_{E_\gamma} = \lambda_\gamma$. Finally, thanks to the uniform control on the gradient of the functions $f_{E_\gamma}$, up to taking smaller $\varepsilon>0$ and $\delta>0$, we have
$$
\Bigl| \int_{E_\gamma^t} f_E\,\mathrm{d}x - \int_{E_\gamma} f_E\,\mathrm{d}x  \Bigr|<\varepsilon\,,
$$
for every $t\in[0,1]$. Thus, we can write
\begin{align*}
\np^\gamma_{f_E,E_\gamma}(F) &- \np^\gamma_{f_E,E_\gamma}(E_\gamma) = \int_0^1 (1-t)\Biggl[ \partial^2\np^\gamma_{E_\gamma}(E_\gamma^t)[X_\gamma\cdot\nu_{E_\gamma^t}] \\
&- \int_{\partial E_\gamma^t}(4\gamma v^{E_\gamma^t}+H_{E^t_\gamma})\div_{\tau_t}(X_\gamma^{\tau_t}(X_\gamma\cdot \nu_{E_\gamma^t})) \dh\\
&-  2 \Bigl( \int_{E^\gamma_t} f_E\,\mathrm{d}x - \int_{E_\gamma} f_E\,\mathrm{d}x \Bigr)\cdot        
                   \int_{\partial E^t_\gamma}f_E\,\div_{\tau_t}(X_\gamma^{\tau_t}(X_\gamma\cdot\nu_{E^t_\gamma}))\dh \Bigg]\mathrm{d}t\,.
\end{align*}
Since the vector fields $X_\gamma$'s are uniformly close in the $C^1$-topology, it is possible to find a constant $C>0$ such that
$$
\|\div_{\tau_t}(X_\gamma^{\tau_t}(X_\gamma\cdot\nu_{E^t_\gamma}))\|_{L^{\frac{p}{p-1}}(\partial E_\gamma^t)}\leq
        C \|X_\gamma\cdot \nu_{E^t_\gamma}\|^2_{H^1(\partial E_\gamma^t)}\,,
$$
for every $\gamma\in(\bar{\gamma}-\widetilde{\gamma}, \bar{\gamma}+\widetilde{\gamma})$. Thus, the above uniform estimates allow us to conclude, as in \cite[Theorem 3.9]{AFM}.
\end{proof}

We now need a technical result that will allow us to obtain the above local minimality property also for the functional $\np^\gamma$.

\begin{lemma}\label{lem:penug}
Let $E$ and $E_\gamma$ as in the statement of Lemma ~\ref{lem:w2p}, and consider the function $f_E$ given by Lemma ~\ref{lem:f}.
Then there exists $\varepsilon>0$ with the following property: for any $F\subset\Pi^N$ with
$\mathrm{d}_{C^1}(E_\gamma,F)<\varepsilon$, there exists $v\in\R^N$ such that
$$
\int_{F+v} f_E\,\mathrm{d}x = \int_{E_\gamma} f_E\,\mathrm{d}x\,.
$$
\end{lemma}

\begin{proof}
\emph{Step 1.} Consider the function $\widetilde{T}:\R^N\rightarrow\R^N$ given by
$$
\widetilde{T}(v):=\int_E f_E(x-v)\mathrm{d}x\,.
$$
Then
$$
D\widetilde{T}(0)=-\int_E Df_E(x)\mathrm{d}x=-\int_{\partial E} \nu\otimes\nu\dh\,.
$$
By \eqref{eq:diag} we know that there exists an orthonormal frame respect to which
\[
\bigl(D\widetilde{T}(0)\bigr)_{ij}=-\int_{\partial E} \nu_i\nu_j\dh=0\,,
\]
if $i\neq j$. Assume $\nu_i\equiv0$ for all $i=1,\dots,k$, for some $k\in\{0,\dots,N-1\}$ (in the case $k=0$, it means that $\nu_i\not\equiv0$ for all $i=1,\dots,N$).
In particular, we have that $f^i_E=0$ for all $i=1,\dots,k$.
Thus, we can consider the following dimensional reduction: define the map $T:\R^{N-k}\to\R^{N-k}$ as
\[
T(v):=\mathrm{P}_{N-k} \left(\,\widetilde{T}(v) \,\right)\,,
\]
where $\mathrm{P}_{N-k}:\R^N\to\R^{N-k}$ is the projection onto the last $N-k$ coordinates.
Then, from the above computations, we get that the matrix $DT(0)$ is invertible.\\

\emph{Step 2.} Fix $\gamma\in(\bar{\gamma}-\widetilde{\gamma}, \bar{\gamma}+\widetilde{\gamma})$, where $\widetilde{\gamma}>0$ is the constant given by Lemma \ref{lem:w2p}.
Consider the map $T_\gamma:\R^{N-k}\to\R^{N-k}$ as
\[
T_\gamma(v):=\mathrm{P}_{N-k} \left(\, \int_{E_\gamma} f_E(x-v)\mathrm{d}x \,\right)\,.
\]
since in Step 1 we have seen that the matrix $D T(0)$ is invertible, it is possible to find $\widetilde{\gamma}>0$ small enough,
such that $D T_\gamma(0)$ in invertible.
This implies that there exist constants $\delta_1, \delta_2>0$ such that
$$
T_\gamma\bigl( B_{\delta_1}(0) \bigr) \supset B_{\delta_2}\bigl( T_\gamma(0) \bigr)\,.\\
$$

One can see that, for any $\varepsilon>0$ small enough, it is possible to find a constant $\widetilde{\varepsilon}>0$ with the following property: if $F\subset\Pi^N$ is such that $\mathrm{d}_{C^1}(E_\gamma,F)<\varepsilon$, then there exists a diffeomorphism $\Phi:E_\gamma\rightarrow F$ of class $C^1$ such that
$\|\Phi-\id\|_{C^1}<\widetilde{\varepsilon}$.
In particular it holds that $\widetilde{\varepsilon}\rightarrow0$ as $\varepsilon\rightarrow0$.

Let $F\subset\Pi^N$ as above and consider the map $T_\gamma^\Phi:\R^N\rightarrow\R^N$ given by
$$
\widetilde{T}_\gamma^\Phi(v):=\int_{E_\gamma} f_E\bigl( \Phi^{-1}(x)-v \bigr)J\Phi(x)\,\mathrm{d}x\,.
$$
For $\varepsilon>0$ small enough, we have that the $i^{th}$ component of
$\widetilde{T}_\gamma^\Phi(0)$ is zero, for all $i=1,\dots,k$.
Thus, as above, we can consider the map $T_\gamma^\Phi:\R^{N-k}\to\R^{N-k}$ via
\[
T_\gamma^\Phi(v):=\mathrm{P}_{N-k}\left(\, \widetilde{T}_\gamma^\Phi(v) \,\right)\,.
\]
Then
$$
DT^\Phi_\gamma(0)=-\int_{E_\gamma} D(\mathrm{P}_{N-k}f_E)\bigl(\Phi^{-1}(x) \bigr)J\Phi(x)\,\mathrm{d}x\,.
$$
Fixed $\mu>0$, there exists $\varepsilon>0$ such that
$$
\|DT_\gamma^\Phi(0) - DT_\gamma(0)\|_{C^{0}}\leq \mu\,,
$$
whenever $\mathrm{d}_{C^1}(E_\gamma,F)<\varepsilon$, and
$\gamma\in(\bar{\gamma}-\widetilde{\gamma}, \bar{\gamma}+\widetilde{\gamma})$.
This follows by using the fact that $\|\Phi-\id\|_{C^1}<\widetilde{\varepsilon}$ and by the uniform control on the $C^1$-norm of the functions $f_\gamma$'s.
Thus, $T_\gamma^\Phi$'s can be made uniformly close to $T_\gamma$ in the $C^1$ topology.

This implies that it is possible to find $\varepsilon>0$ such that if
$\mathrm{d}_{C^1}(E_\gamma,F)<\varepsilon$, then
\begin{equation}\label{eq:sup}
T_\gamma^\Phi(B_{\delta_1/2}(0))\supset B_{\delta_2/2}(T_\gamma^\Phi(0))\,.
\end{equation}
This follows, for instance, from the proof of the Inverse Function Theorem.\\

\emph{Step 3.} We can now easily conclude as follows: up to taking a smaller $\varepsilon$, we can suppose
$T_\gamma^\Phi(0)\in B_{\delta_2/4}(T(0))$, whenever $\mathrm{d}_{C^1}(E_\gamma,F)<\varepsilon$.
Thus, by \eqref{eq:sup}, we have that there exists $v\in B_{\delta_1/2}(0)$ such that $T_\gamma^\Phi(v)=T_\gamma(0)$.
This is exactly the statement we wanted to prove.\\
\end{proof}

\begin{lemma}\label{lem:minw2p}
Take $p>\max\{2,N-1\}$, $E\subset\Pi^N$ be a strictly stable critical point for $\np^{\bar{\gamma}}$, and let $\widetilde{\gamma}>0$, $\delta>0$ and $C>0$ be the constants given by Lemma \ref{lem:w2p}.
Then, for any
$\gamma\in(\bar{\gamma}-\widetilde{\gamma},\bar{\gamma}+\widetilde{\gamma})$ and $E_\gamma$ critical point for $\np^\gamma$ with $\mathrm{d}_{C^1}(E,E_\gamma)<\varepsilon$, we have that
$$
\np^\gamma(F)\geq \np^\gamma(E_\gamma) + C \bigl(\alpha(E_\gamma, F )\bigr)^2\,,
$$
for every set $F\subset\Pi^N$ with $|F|=|E_\gamma|$ and $\partial F=\{ x+\psi(x)\nu_{E_\gamma}(x) \;:\; x\in\partial E_\gamma \}$ for some
$\|\psi\|_{W^{2,p}(\partial E_\gamma)}\leq\delta$.
\end{lemma}

\begin{proof}
Fix a number $\varepsilon\in(0,\widetilde{\varepsilon})$, where $\widetilde{\varepsilon}>0$ is the constant given by Lemma \ref{lem:penug}. Then we know that we can find a vector $v\in\R^N$ such that
$$
\mathrm{Pen}_{f_E,E_\gamma}(F+v)=0\,.
$$
Thus, by using the result of Lemma ~\ref{lem:w2p}, we can write
\begin{align*}
\np^\gamma(F) &= \np^\gamma(F+v) = \np^\gamma_{f_E,E_\gamma}(F+v)\geq \np^\gamma_{f_E,E_\gamma}(E_\gamma) + C |E_\gamma\triangle F|^2\\
&\geq \np^\gamma(E_\gamma) + C \bigl(\alpha(E_\gamma, F) \bigr)^2\,,
\end{align*}
where the last inequality follows from Definition \ref{def:alpha}.
\end{proof}

We now prove the uniform $L^\infty$-local minimality result, \emph{i.e.}, the uniform version of \cite[Theorem 4.3]{AFM}.

\begin{lemma}
Let $E\subset\Pi^N$ be a strictly stable critical point for $\np^{\bar{\gamma}}$.
Then there exist constants $\delta>0$, $\widetilde{\gamma}>0$ and $\varepsilon>0$ with the following property: for any $\gamma\in(\bar{\gamma}-\widetilde{\gamma},\bar{\gamma}+\widetilde{\gamma})$ and any $E_\gamma$ critical point for $\np^\gamma$ with $\mathrm{d}_{C^1}(E,E_\gamma)<\varepsilon$, it holds
$$
\np^\gamma(E_\gamma)\leq\np^\gamma(F)\,,
$$
for every set $F\subset\Pi^N$ with $|F|=|E_\gamma|$, such that $E_\gamma\triangle F\Subset \mathcal{N}_\delta(E_\delta)$, where
$\mathcal{N}_\delta(E_\gamma)$ is a tubular neighborhood of $\partial E_\gamma$ of thickness $\delta$.
\end{lemma}

\begin{proof}
Suppose for the sake of contradiction that there exist a sequence $\gamma_n\rightarrow\bar{\gamma}$, $E_{\gamma_n}\rightarrow E$ in $C^1$, with
$|E_\gamma|=|E|$, a sequence $\delta_n\rightarrow0$ and a sequence
of sets $F_n$ with $|F_n|=|E_{\gamma_n}|$, $E_\gamma\triangle F_n\Subset \mathcal{N}_\delta(E_{\gamma_n})$, such that
$$
\np^{\gamma_n}(E_{\gamma_n})>\np^{\gamma_n}(F_n)\,.
$$
Let $E_n$ be a solution of the following constrained minimum problem
$$
\min\bigl\{ \np^{\gamma_n}(F) + \Lambda\bigl| |F|-|E_\gamma| \bigr| \;:\; F\triangle E_\gamma \subset \mathcal{N}_\delta(E_{\gamma_n}) \bigr\}\,.
$$
By using the $C^{3,\beta}$ convergence of the $E_{\gamma_n}$'s to $E$, and reasoning as in the proof of \cite[Theorem 4.3]{AFM}, it is possible to find a constant $\Lambda>0$ independent of $\gamma_n$ such that the sets $E_n$'s are $(4\Lambda, r_0)$-minimizers of the area functional, for some $r_0>0$ independent of $\gamma_n$, and $|E_n|=|E_\gamma|$. This is because, if we set $\nu_n:=\nabla d_n$ (defined in $(\partial E)_\mu$, for some $\mu>0$), where $d_n$ is the signed distance from $E_n$, we have that $\|\div\,\nu_n\|_{L^\infty}\leq C$ for some constant $C>0$ independent of $n$.

Since $(E_n)_n$ is a sequence of uniform $(\omega, r)$-minimizers converging to $E$ in the $L^1$ topology, by Theorem \ref{thm:W} we have that, indeed,
$E_n\rightarrow E$ in the $W^{2,p}$-topology. By using again the $C^{3,\beta}$ convergence of the $E_{\gamma_n}$'s to $E$ and the Euler-Lagrange equation satisfied by each $E_n$, we obtain that $\mathrm{d}_{W^{2,p}}(E_n, E_{\gamma_n})\rightarrow0$ as $n\rightarrow\infty$. Since, by definition,
$\np^\gamma(E_n)<\np^\gamma(E_{\gamma_n})$we obtain a contradiction with the result of Lemma ~\ref{lem:minw2p}.
\end{proof}


\begin{proof}[Proof of Proposition~\ref{prop:unif}]
Suppose for the sake of contradiction that there exists a sequence $\gamma_n\rightarrow\bar{\gamma}$, $E_{\gamma_n}\rightarrow E$ in $C^1$, with
$|E_\gamma|=|E|$, a sequence $\delta_n\rightarrow0$ and a sequence of sets $F_n$ with $|F_n|=|E_{\gamma_n}|$, and $0<\varepsilon_n\rightarrow0$,
where $\varepsilon_n:=\alpha(F_n, E_{\gamma_n})$, such that
$$
\np^{\gamma_n}(F_n)\leq \np^{\gamma_n}(E_{\gamma_n})+\frac{C}{4}\bigl( \alpha(E_{\gamma_n},F_n) \bigr)^2\,.
$$
Let $E_n$ be a solution of the following constrained minimum problem
$$
\min\bigl\{ \np^{\gamma_n}(F) + \Lambda\sqrt{\bigl(\alpha(F,E_{\gamma_n})-\varepsilon_n\bigr)^2+\varepsilon_n} \;:\; |F|=|E_\gamma| \bigr\}\,.
$$
Then, by using a $\Gamma$-convergence argument it is possible to prove that the $E_n$'s converge (up to a subsequence) in the $L^1$ topology to a solution of the limiting
problem
$$
\min\bigl\{ \np^{\bar{\gamma}}(F) + \Lambda|\alpha(F,E)| \;:\; |F|=|E| \bigr\}\,.
$$
Reasoning as in the proof of \cite[Theorem 1.1]{AFM} and by using the $C^{3,\beta}$ convergence of the $E_{\gamma_n}$'s to $E$ (see Lemma \ref{lem:convc3}), it is possible to prove that there exists a constant $\Lambda>0$ such that the unique solution of the limiting problem is $E$ itself.
Moreover, by reasoning again as in the proof of \cite[Theorem 1.1]{AFM} and using Lemma~\ref{lem:qm}, we can also infer that $(E_n)_n$ is a sequence of uniform
$(\omega,r)$-minimizers, and that $E_n\rightarrow E$ in the $W^{2,p}$-topology. Thus, $\mathrm{d}_{W^{2,p}}(E_n, E_{\gamma_n})\rightarrow0$ as $n\rightarrow\infty$.
Using the previous uniform $L^\infty$-local minimality result is it also possible to prove that
$\frac{\alpha(E_n,E_{\gamma_n})}{\alpha(F_n,E_{\gamma_n})}\rightarrow 1$ (see \cite[equation (4.17)]{AFM}).
Thus, we may conclude
\begin{align*}
\np^{\gamma_n}(E_n)&\leq \np^{\gamma_n}(F_n)\leq \np^{\gamma_n}(E_{\gamma_n})+\frac{C}{4}\bigl( \alpha(E_{\gamma_n},F_n) \bigr)^2\\
&\leq\np^{\gamma_n}(E_{\gamma_n})+\frac{C}{2}\bigl( \alpha(E_{\gamma_n},E_n) \bigr)^2\,.
\end{align*}
This yelds the contradiction with the result of Lemma~\ref{lem:minw2p}.\\
\end{proof}


\subsection{Continuous family of local minimizers}

We now prove a uniqueness result for critical points of $\np^\gamma$ close enough to a regular critical stable point of the area functional.
We also prove that these critical points are isolated local minimizers.

\begin{proposition}\label{prop:uniflocmin}
Let $\bar{\gamma}\geq0$ and let $E\subset\Pi^N$ be a strictly stable critical point for $\np^{\bar{\gamma}}$.
Then there exist constants $\widetilde{\gamma}>0$ and $\varepsilon>0$ and a unique family
$\gamma\mapsto E_\gamma$, for
$\gamma\in(\bar{\gamma}-\widetilde{\gamma},\bar{\gamma}+\widetilde{\gamma})$, with $|E_\gamma|=|E|$, such that
\begin{itemize}
\item $\mathrm{d}_{C^1}(E_\gamma, E)<\varepsilon$,
\item $E_\gamma$ is a critical point for $\np^\gamma$.
\end{itemize}
Moreover $\gamma\mapsto E_\gamma$ is continuous in $C^{3,\beta}$, for all $\beta\in(0,1)$, and $E_\gamma$ is an isolated local minimizer of $\np^\gamma$.
\end{proposition}

\begin{proof}
\emph{Step 1: existence of the family}. First of all we notice that, by Theorem ~\ref{thm:locmin}, we can find a constant $\delta>0$ such that
$$
\np^{\bar{\gamma}}(E)<\np^{\bar{\gamma}}(F)\,,
$$
for any set $F\subset\Pi^N$ with $|F|=|E|$, such that $0<\alpha(E,F)<\delta$.
Then it is possible to use Theorem ~\ref{thm:KS} to find a sequence $(E_\gamma)_\gamma$, with $|E_\gamma|=|E|$, such that
$E_\gamma$ is a local minimizer of $\np^\gamma$, and $\alpha(E_\gamma,E)\rightarrow0$ as $\gamma\rightarrow\bar{\gamma}$.

By using Corollary ~\ref{cor:qm}, we infer that the sequence $(E_\gamma)_\gamma$ is a sequence of $(\omega_0, r_0)$-minimizers, where the parameter $\omega_0$ can be chosen uniformly with respect to $\gamma$
(see Lemma ~\ref{lem:qm}).
Hence, Theorem ~\ref{thm:W} allows to say that the $E_\gamma$'s actually converge to $E$ in the $C^{1,\beta}$-topology.\\

\emph{Step 2: uniqueness of the family}. Let $\varepsilon_0>0$ and $\gamma_0>0$ be the constants given by Proposition \ref{prop:unif}, and
take $\varepsilon<\varepsilon_0$ and $\widetilde{\gamma}<\gamma_0$ such that
$$
\mathrm{d}_{C^1}(E_\gamma,E)<\varepsilon\,,
$$
for any $\gamma\in(\bar{\gamma}-\widetilde{\gamma},\bar{\gamma}+\widetilde{\gamma})$.
By Proposition ~\ref{prop:unif} there exists $\delta>0$ such that the $E_\gamma$'s are uniform local minimizers with respect to sets $F$ with $|F|=|E_\gamma|$ with
$\alpha(F,E_\gamma)\leq\delta$. In particular, we have that
$$
\np^\gamma(E_\gamma)<\np^\gamma(F)\,,
$$
for any set $F\neq E_\gamma$ with $|F|=|E_\gamma|$ and $\alpha(F,E_\gamma)\leq\delta$.

By taking a smaller $\varepsilon$ (and a smaller $\widetilde{\gamma}$) if necessary, we can assume that
$$
\mathrm{d}_{C^1}(F,E)<\varepsilon \,\Rightarrow\, \alpha(F,E_\gamma)\leq\delta\,,
$$
for any set $F\subset\Pi^N$ and any $\gamma\in(\bar{\gamma}-\widetilde{\gamma},\bar{\gamma}+\widetilde{\gamma})$.
This allows to infer that $E_\gamma$ is the unique critical point of $\np^\gamma$ with $|E_\gamma|=|E|$ and $\mathrm{d}_{C^1}(E_\gamma,E)<\varepsilon$. Indeed, if $F$ is another critical point of $\np^\gamma$, with $|F|=|E|$ with $\mathrm{d}_{C^1}(F,E)<\varepsilon$, by using again
Proposition~\ref{prop:unif}, we would obtain that $F$ is an isolated local minimizer of $\np^\gamma$ with respect to sets $G$ with $|G|=|F|$ and
$\alpha(G,F)\leq\delta$. But this is in contradiction with the isolated local minimality property of $E_\gamma$.\\

\emph{Step 3: continuity}. Finally, we can deduce the continuity in the $C^{3,\beta}$-topology of the family $\gamma\mapsto E_\gamma$ as follows:
fix $\gamma\in(\bar{\gamma}-\widetilde{\gamma},\bar{\gamma}+\widetilde{\gamma})$, and let $\gamma_n\rightarrow\gamma$.
Then, up to a subsequence, there exists a set $F\subset\Pi^N$ such that $E_{\gamma_n}\rightarrow F$ in the $L^1$ topology. By the uniqueness property just proved, we have that $F=E_\gamma$.

Moreover, since $(E_{\gamma_n})_n$ is a sequence of uniform $(\omega, r_0)$-minimizers, we can use Lemma~\ref{thm:W} to infer that $E_{\gamma_n}\rightarrow F$ in the $C^{1,\beta}$ topology.
Thus, by using Lemma \ref{lem:convc3} we obtain the convergence of $E_{\gamma_n}$ to $E_\gamma$ in the $C^{3,\beta}$-topology.
\end{proof}


\subsection{Periodic local minimizers with almost constant mean curvature}

The main result of this chapter is the following.

\begin{theorem}\label{thm:main}
Let $E\subset\Pi^N$ be a smooth set that is critical and strictly stable for the area functional, \emph{i.e.}, there exists $\lambda\in\R$ such that
$$
H_{\partial E}=\lambda\quad\quad\text{ on } \partial E\,,
$$
and
$$
\int_{\partial E} \bigl( |D_{\tau}\varphi|^2-|B_{\partial E}|^2\varphi^2 \bigl)\dh>0\quad\quad
                 \text{ for every } \varphi\in T^\bot(\partial E)\meno\{0\}\,.
$$
Fix constants $\bar{\gamma}>0$, $\varepsilon>0$.
Then it is possible to find $\bar{k}=\bar{k}(\bar{\gamma}, \varepsilon)\in\N$ and $C=C(\bar{\gamma})>0$ such that for all $k\geq\bar{k}$ there exists a unique set $F\subset\Pi^N$ that is $1/k$-periodic and with
\begin{itemize}
\item $\dist_{C^0}(F, E^k)<\frac{\varepsilon}{k}$, where $E^k$ is as Definition~\ref{def:per},
\item $\dist_{C^1}(F, E^k)<\varepsilon$,
\item $\|\nabla_\tau H_F\|_{L^\infty(\partial F)}<\frac{C}{k}$, where $H_F$ is the mean curvature of $\partial F$.
\end{itemize}
Moreover $F$ is an isolated local minimizer of $\np^{\bar{\gamma}}$ with respect to $1/k$-periodic sets, \emph{i.e.}, there exists $\delta>0$ such that,
for any set $G\subset\Pi^N$ that is $1/k$-periodic and with $|G|=|F|$, it holds
$$
\np^{\bar{\gamma}}(F)<\np^{\bar{\gamma}}(G)\,,
$$
whenever $0<\alpha(G,F)\leq\delta$.
\end{theorem}

\begin{proof}
Consider the sequence
$$
(\gamma_k)_k:=(\bar{\gamma}k^{-3})_{k\in\N\setminus\{0\}}\,.
$$
Let $\gamma_k\mapsto E_{\gamma_k}$ be the unique family provided by Proposition~\ref{prop:uniflocmin} applied to $E$.
Take $\bar{k}$ such that, for all $k\geq\bar{k}$, $\dist_{C^1}(E_{\gamma_k}, E)<\varepsilon$ and $E_{\gamma_k}$ is an isolated local minimizer of
$\np^{\gamma_k}$. This can be done by using the results of Proposition~\ref{prop:uniflocmin}. Let $F:=E^k_{\gamma_{k}}$. Now, it is easy to see that
$$
\dist_{C^0}(F, E^k) = \frac{1}{k}\dist_{C^0}(E_{\gamma_k}, E)<\frac{\varepsilon}{k}\,,\quad \quad
          \dist_{C^1}(F, E^k) = \dist_{C^1}(E_{\gamma_k}, E) <\varepsilon\,.
$$
Moreover, by \eqref{eq:per} and \eqref{eq:pfk}, we have that
$$
\np^{\bar{\gamma}}(F) = k^N\np_k^{\bar{\gamma}}(E_{\gamma_{k}}) =
               k\bigl[ \per_{\Pi^N}(E_{\gamma_k}) + \gamma_k\nl_{\Pi^N}(E_{\gamma_k}) \bigr] = k\np^{\gamma_k}_{\Pi^N}(E_{\gamma_k})\,.
$$
Since $E_{\gamma_k}$ is an isolated local minimizer for $\np^{\gamma_k}$, we obtain that $F$ satisfied the isolated local minimimality property of the theorem.

Finally, we have that
$$
H_{\partial F}(x) = k H_{\partial E_{\gamma_k}}(kx) = k \bigl( \lambda_k - 4\gamma_k v^{E_{\gamma_k}}(kx) \bigr)\,,
$$
where in the last step we have used the Euler-Lagrange equation satisfied by $E_{\gamma_k}$. Thus, using the definition of $\gamma_k$, we obtain that
$$
\|\nabla_\tau H_F\|_{L^\infty(\partial F)} \leq \frac{4\bar{\gamma}}{k}\|\nabla v^{E_{\gamma_k}}\|_{L^{\infty}(\partial E_{\gamma_k})}\,.
$$
Since $v^{E_{\gamma_k}}\rightarrow v^E$ in $C^{1,\beta}$, up to choose a bigger $\bar{k}$, we also have the desired estimate for
$\|\nabla_\tau H_F\|_{L^\infty(\partial F)}$.
\end{proof}

\begin{figure}[H]
\includegraphics[scale=0.35]{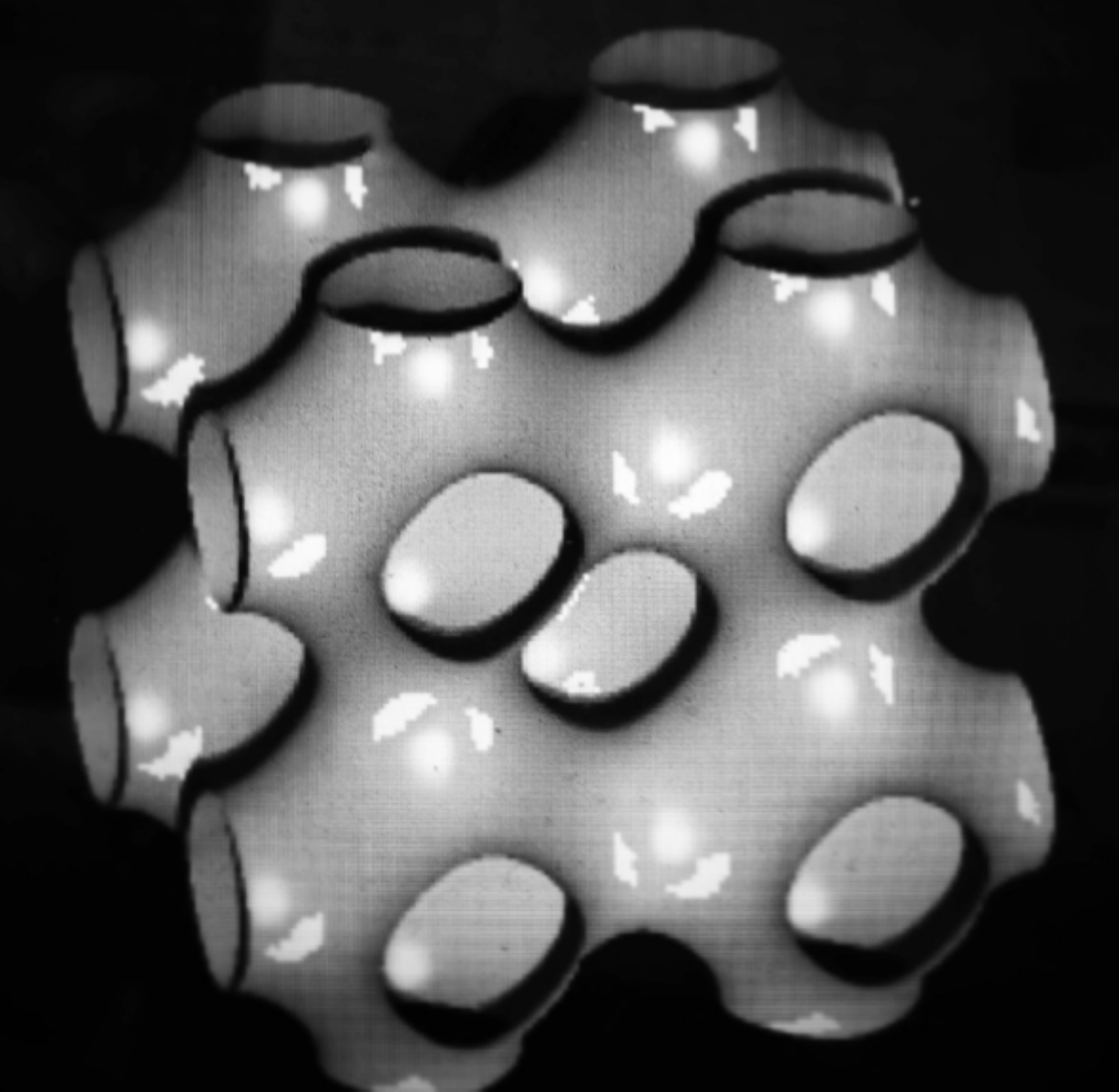}
\caption{An example of a strictly stable periodic surface with constant mean curvature.}
\end{figure} 

We finally show that the critical points constructed in the above theorem can be approximated with local minimizers of the
$\varepsilon$-diffuse energy $OK^{\bar{\gamma}}_\varepsilon$.

\begin{corollary}\label{prop:approx}
Let $E\subset\Pi^N$ be as in the previous theorem, and let $F$ be a periodic critical point constructed above.
Define the function $u:=\chi_F-\chi_{\Pi^N\meno F}$.
Then there exist a constant $\bar{\varepsilon}>0$ and a family
$(u_\varepsilon)_{\varepsilon\in(0,\bar{\varepsilon})}$ such that
\begin{itemize}
\item $u_\varepsilon$ is a local minimizer of $OK^{\bar{\gamma}}_\varepsilon$,
\item $\int_{\Pi^N}u_\varepsilon = \int_{\Pi^N}u$,
\item $u_\varepsilon\rightarrow u$ in $L^1(\Pi^N)$ as $\varepsilon\rightarrow0$.
\end{itemize}
\end{corollary}

\begin{proof}
The proof follows by Kohn and Sternberg's theorem (see \cite{KohSter} and also \cite[Proposition 8]{ChoSte1}), thanks to the $\Gamma$-convergence of $OK^{\bar{\gamma}}_\varepsilon$ to $\np^{\bar{\gamma}}$ and using the fact that $F$ is an isolated local minimizer with respect to
$1/k$-periodic perturbations.
\end{proof}

\begin{remark}\label{rem:strngemin}
The local minimality provided in Theorem \ref{thm:main} requires the variations to be $1/k$-periodic.
But actually, it is possible to prove a slightly more general local minimality property holds true.
The statement is the following:

let $E\subset\Pi^N$ be a smooth set that is critical and strictly stable for the area functional.
Fix constants $\bar{\gamma}>0$, $\varepsilon>0$.
Then it is possible to find $\bar{k}=\bar{k}(\bar{\gamma}, \varepsilon)\in\N$ and $C=C(\bar{\gamma})>0$ such that for all $k\geq\bar{k}$ there exists a unique set $F\subset\Pi^N$ that is $1/k$-periodic and with
\begin{itemize}
\item $\dist_{C^0}(F, E^k)<\frac{\varepsilon}{k}$, where $E^k$ is as Definition~\ref{def:per},
\item $\dist_{C^1}(F, E^k)<\varepsilon$,
\item $\|\nabla_\tau H_F\|_{L^\infty(\partial F)}<\frac{C}{k}$, where $H_F$ is the mean curvature of $\partial F$.
\end{itemize}

Moreover the following isolated local minimality property holds true:
there exist constants $\delta>0$ and $D>0$ such that
$$
\np^{\bar{\gamma}}(F)+D\bigl( \alpha(G,F) \bigr)^2\leq\np^{\bar{\gamma}}(G)\,,
$$
for every set $G\subset\Pi^N$ having $|G|=|F|$ that satisfies
$$
\partial G=\{ x+\Psi(x)\nu_F(x) \,:\, x\in\partial F \}\,,
$$
where $\Psi\in W^{2,p}(\partial F)$ is such that:
\begin{itemize}
\item $\|\Psi\|_{W^{2,p}(\partial F)}\leq\delta$,
\item $G$ restricted to every $1/k$-periodicity cell has the same volume of the set $F$ restricted to the same periodicity cell,
\item the restriction of $\Psi$ on each $1/k$-periodicity cell is $1/k$-periodic.
\end{itemize}

\begin{figure}
\includegraphics[scale=1.2]{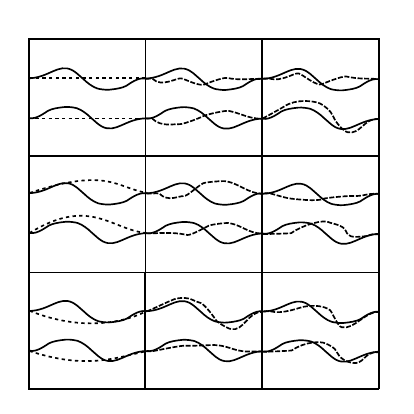}
\caption{An example of a $1/k$-periodic set $F$ (bold lines denotes $\partial F$) and of an admissible competitor $G$ (dotted lines denotes $\partial G$).}\label{fig:competitor}
\end{figure}

To prove it, we reason as follows: take the family $(E_{\gamma_k})_k$ constructed in the proof of Theorem \ref{thm:main}. Then, it holds that
$v_{E^k_{\gamma_{k}}}(x)=k^{-2}v_{E_{\gamma_{k}}}(kx)$, and thus
\begin{equation}\label{eq:conv}
\|\nabla v_{E^k_{\gamma_{k}}}\|_{L^\infty}\rightarrow 0\,,
\end{equation}
as $k\rightarrow\infty$.
Now take a function $\varphi\in T^\bot(\partial E^k_{\gamma_{k}})$ with zero average in each periodicity cell and such that the restriction of $\varphi$ on each $1/k$-periodicity cell is $1/k$-periodic. We would like to compute the second variation of $\mathcal{F}^{\bar{\gamma}}$ computed at $E^k_{\gamma_{k}}$, in $\varphi$.
Since the sets $E^k_{\gamma_{k}}$'s are critical, this is given by \eqref{eq:pfquad}:
\[
\begin{split}
\int_{\partial E^k_{\gamma_{k}}}& \bigl( |D_{\tau}\varphi|^2-|B_{\partial E^k_{\gamma_{k}}}|^2\varphi^2 \bigl) \,\dh
    +4\bar{\gamma}\int_{\partial E^k_{\gamma_{k}}} (\partial_{\nu_{E^k_{\gamma_{k}}}}v^{E^k_{\gamma_{k}}}) \varphi^2\,\dh \\
& \quad +8\bar{\gamma}\int_{\partial E^k_{\gamma_{k}}}\int_{\partial E^k_{\gamma_{k}}}G_{\Pi^N}(x,y) \varphi(x) \varphi(y)\,\dh(x)\dh(y)\,.
\end{split}
\]
Notice that:
\begin{itemize}
\item the first term is strictly positive for $k$ large: indeed, since $\varphi$ satisfies the two conditions above, this follows by using a rescaling argument, the fact that $E_{\gamma_{k}}\rightarrow E$ in
$C^{3,\beta}$ and that $E$ is strictly stable for the area functional,
\item the second term is uniformly small with respect to $\varphi$, by \eqref{eq:conv},
\item the last term is non-negative, since it can be written as in \eqref{eq:write}.
\end{itemize}

Thus, we have that, for $k$ large enough, the sets $E^k_{\gamma_{k}}$'s are strictly stable with respect to this kind of admissible functions $\varphi$'s.
Thus, it is possible to reasoning as in the proof of \cite[Theorem~3.9]{AFM} in order to obtain the claimed local minimality property.
\end{remark}


\bigskip
\bigskip
\noindent
{\bf Acknowledgments.}
{The author wishes to thank Massimiliano Morini for having introduced him to the study of this problem and for multiple helpful discussions we had during the preparation of the paper, as well as the referees for having proposed several suggestions to improve the paper.

The author thanks SISSA and the Center for Nonlinear Analysis at Carnegie Mellon University for their support during the preparation of the manuscript.
The research was partially supported by National Science Foundation under Grant No. DMS-1411646.
The results contained in the paper are part of the Ph.D. thesis of the author.
}


\bibliographystyle{siam} 
\bibliography{bibliografia}

\end{document}